\documentclass[12pt,a4paper,reqno,oneside]{amsart}

\usepackage{amsmath,amssymb,amsthm,graphicx,eucal,url}

\numberwithin{equation}{section}
\theoremstyle{plain}
\newtheorem{thm}{Theorem}[section]
\newtheorem{lem}[thm]{Lemma}

\newtheorem{prop}[thm]{Proposition}
\newtheorem{cor}[thm]{Corollary}

\theoremstyle{remark}
\newtheorem{remark}[thm]{Remark}

\theoremstyle{plain}

\DeclareMathOperator{\vol}{vol}

\newcommand{\ess}{\mathrm{ess}}

\newcommand{\RR}{\mathbb{R}}
\newcommand{\CC}{\mathbb{C}}
\newcommand{\NN}{\mathbb{N}}

\newcommand{\bfB}{\mathbf{B}}

\newcommand{\cB}{\mathcal{B}}
\newcommand{\cC}{\mathcal{C}}
\newcommand{\cF}{\mathcal{F}}
\newcommand{\cH}{\mathcal{H}}
\newcommand{\cG}{\mathcal{G}}

\newcommand{\cL}{\mathcal{L}}
\newcommand{\cN}{\mathcal{N}}
\newcommand{\cO}{\mathcal{O}}
\newcommand{\cP}{\mathcal{P}}

\newcommand{\cR}{\mathcal{R}}
\newcommand{\cT}{\mathcal{T}}

\newcommand{\dd}{\mathrm{d}}
\newcommand{\rmi}{\mathrm{i}}

\newcommand{\dom}{\mathop{\mathcal{D}}}
\newcommand{\ran}{\mathop{\mathrm{ran}}}

\DeclareMathOperator{\dist}{dist}
\DeclareMathOperator{\ddiv}{div}
\newcommand{\spec}{\mathop{\mathrm{spec}}\nolimits}
\newcommand{\res}{\mathop{\mathrm{res}}\nolimits}

\usepackage[usenames,dvipsnames]{color}

\begin{document}

\title[]
{Dirac operators\\ with Lorentz scalar shell interactions}

\author[]{Markus Holzmann}
\address{Institut f\"{u}r Numerische Mathematik, Technische Universit\"{a}t Graz, Steyrergasse 30, 8010 Graz, Austria}
\email{holzmann@math.tugraz.at}
\urladdr{http://www.math.tugraz.at/~holzmann/}

\author[]{Thomas Ourmi\`eres-Bonafos}
\address{Laboratoire de Math\'ematiques d'Orsay, Univ.~Paris-Sud, CNRS, Universit\'e Paris-Saclay, 91405 Orsay, France}
\email{thomas.ourmieres-bonafos@math.u-psud.fr}
\urladdr{http://www.math.u-psud.fr/~ourmieres-bonafos/}

\author[]{Konstantin Pankrashkin} 
\address{Laboratoire de Math\'ematiques d'Orsay, Univ.~Paris-Sud, CNRS, Universit\'e Paris-Saclay, 91405 Orsay, France}
\email{konstantin.pankrashkin@math.u-psud.fr}
\urladdr{http://www.math.u-psud.fr/~pankrashkin/}

\date{\today}

\begin{abstract}
This paper deals with the massive three-dimensional Dirac operator coupled with a Lorentz scalar shell interaction supported on a compact smooth surface.
The rigorous definition of the operator involves suitable transmission conditions along the surface.
After showing the self-adjointness of the resulting operator we switch to the investigation of its spectral properties, in particular, to the existence and non-existence of eigenvalues.
In the case of an attractive coupling, we study the eigenvalue asymptotics as the mass becomes
large and show that the behavior of the individual eigenvalues and their total number
are governed by an effective Schr\"odinger operator on the boundary
with an external Yang-Mills potential
and a curvature-induced potential.
\end{abstract}

\maketitle

\tableofcontents

\section{Introduction}

\subsection{Motivations and main results}

The Dirac operator was introduced to give a quantum mechanical framework that takes relativistic properties of particles of spin $\frac{1}{2}$ into account. This operator can be seen as a relativistic counterpart of the Schr\"odinger operator and, as for this latter, the behavior of physical systems can be deduced from a thorough spectral analysis \cite{T92}.

In the present paper we focus on a class of Dirac operators with potentials supported on zero measure sets (the so-called $\delta$-potentials). Such interactions are often used in mathematical physics as idealizations for regular potentials located in a neighborhood of this zero set.
While such operators are well understood in the one-dimensional case, see e.g. \cite{AGHH,CMR,GS,PR}
as well as for the closely related radial mutidimensional case \cite{DES89}, the systematic study in higher dimension appeared
to be much more involved and attracted a lot of attention recently.
It seems that the first results on Dirac operators with interactions supported on general smooth surfaces (shells)
were obtained in \cite{AMV14,AMV15,AMV15bis}, where the self-adjointness and the discrete spectrum were discussed.
The analysis was based mostly on the usage of potential operators involving the fundamental solution of the unperturbed Dirac equation.
In \cite{BEHL17,BH17,OV16}, the study was pushed further in order to understand the Sobolev regularity of functions in the domain,
the $\delta$-shell potential being then encoded by a transmission condition at the shell. 
Furthermore, as for Schr\"odinger operators with $\delta$-potentials \cite{BEHL17_1}, the shell interactions
in the Dirac setting can be understood as suitable limits of regular potentials localized near the surface, as it was shown recently
in \cite{MP16,MP17}. One of the main motivations for the present paper is the recent work \cite{ALTR16}, where the closely related MIT bag model dealing with Dirac operators in bounded domains and special boundary conditions were studied. In fact, it is shown in \cite{ALTR16} that for large negative masses the asymptotics
of the MIT bag eigenvalues is determined by an effective operator acting on the boundary, and it is one of our objectives to
study the related problem for scalar shell interactions.

We are going to study the specific case of the three-dimensional Dirac operator coupled with a Lorentz scalar shell interaction of strength $\tau\in\RR$
supported on a smooth compact surface $\Sigma$. The operator acts in $L^2(\RR^3,\mathbb{C}^4)$ and writes formally as
\begin{equation}\label{eqn:op_form}
	A_{m,\tau}:= -\rmi\big(\alpha_1\partial_1+\alpha_2\partial_2+\alpha_3\partial_3\big) +m\beta + \tau \beta\delta_{\Sigma},
\end{equation}
where $\alpha_1,\alpha_2,\alpha_3,\beta$ are the standard $\mathbb{C}^{4\times4}$ Dirac matrices written down explicitly in \eqref{def_Dirac_matrices},
$m\in\mathbb{R}$ is the mass of the particle and $\delta_\Sigma$ is the Dirac distribution on $\Sigma$.
The expression \eqref{eqn:op_form} is formal due to the presence of the singular term $\delta_\Sigma$,
the rigorous definition of $A_{m,\tau}$ is given below in \eqref{equation_dom_B_tau} using
suitable transmission conditions at $\Sigma$.
We remark that the special value $\tau=0$ corresponds to the free Dirac operator, whose properties are well known (see Section~\ref{section_free_Operator}).
Furthermore, the  values $\tau=\pm 2$ play a special role as they
correspond to ``hard walls'' at $\Sigma$, i.e. $A_{m,\pm2}$ is decoupled and represent the direct sum  of two
operators acting inside and outside of $\Sigma$; this corresponds to the so-called MIT bag model already considered in \cite{ALTR16},
see Remark~\ref{rem-mit} below. 

In what follows we exclude the above special values of $\tau$. Our main results can be roughly summed up as follows.
\begin{itemize}
\item[(A)] The operator $A_{m,\tau}$ defined as in \eqref{equation_dom_B_tau} below is self-adjoint,
its spectrum is symmetric with respect to $0$, and its essential spectrum is $(-\infty,-|m|]\cup[|m|,+\infty)$.
\item[(B)] The operator $A_{m,\tau}$ is unitarily equivalent to $A_{m,\frac{4}{\tau}}$ and to $A_{-m,-\tau}$.
\end{itemize}
In view of the preceding symmetry,  without loss of generality for the subsequent points we assume that $m\ge 0$.
\begin{itemize}
\item[(C)] If $\tau \geq 0$, the discrete spectrum of $A_{m,\tau}$ is empty.
\item[(D)] For any $m>0$ there exists $\tau_m>0$ such that the discrete spectrum of $A_{m,\tau}$ 
is empty for $|\tau|<\tau_m$ and for $|\tau|>\dfrac{4}{\tau_m}$.
\end{itemize}
Finally, being motivated by the analysis of \cite{ALTR16} we provide an asymptotic study of the discrete spectrum for the case when
\begin{equation}
    \label{eq-as000}
\tau<0 \text{ with }\tau\ne -2\text{ is fixed }, \quad m\to +\infty
\end{equation}
and obtain the following results:
\begin{itemize}
\item[(E)] The total number of discrete eigenvalues of $A_{m,\tau}$ counted with multiplicities
obeys a Weyl-type law and behaves as
$$
	\frac{16}{\pi} \frac{\tau^2}{(\tau^2 + 4)^2}|\Sigma|m^2+\cO(m\log m),
$$ 
with $|\Sigma|$ being the surface area of $\Sigma$.
\item[(F)] Denote the eigenvalues of $A_{m,\tau}$ by $\pm \mu_j(m)$ with $\mu_j(m)\ge 0$ enumerated in the non-decreasing order,
then for each fixed $j\in\NN$ there holds
\begin{equation}
    \label{eq-eig00}
	\mu_j(m) = \frac{|\tau^2 -4|}{\tau^2+4}m + \frac{\tau^2+4}{|\tau^2 - 4|}\frac{E_j(\Upsilon_\tau)}{2m} + \mathcal{O}\Big(\frac{\log m}{m^2}\Big),
\end{equation}
where $E_j(\Upsilon_\tau)$ is the $j$-th eigenvalue of the $m$-independent Schr\"odinger operator $\Upsilon_\tau$
with an external Yang-Mills potential in $L^2(\Sigma,\CC^2)$,
\[
\Upsilon_\tau = \Big(\dd + \rmi \dfrac{4}{\tau^2+4}\,\omega\Big)^*\Big(\dd + \rmi
\dfrac{4}{\tau^2+4}\,\omega\Big) - \Big(\dfrac{\tau^2-4}{\tau^2+4}\Big)^2 M^2 + \dfrac{\tau^4+16}{(\tau^2+4)^2} K,
\]
where $K$ and $M$ are respectively the Gauss and mean curvature and the $1$-form $\omega$ is given by the local expression
$\omega:=\sigma\cdot (\nu\times \partial_1 \nu)\dd s_1+\sigma\cdot (\nu\times \partial_2 \nu)\dd s_2$
with $\nu$ being the outer unit normal on $\Sigma$. (The precise definition of $\Upsilon_\tau$ is given in Subsection~\ref{sec-eff}.)
\end{itemize}

We remark that by setting formally $\tau=\pm 2$ in \eqref{eq-eig00} one recovers the eigenvalue asymptotics for the MIT bag model as obtained in \cite[Thm.~1.13]{ALTR16}
with the effective operator written in an alternative way.

Let us describe the structure of the paper. In the following Section \ref{ssec-not}
we introduce first a couple of conventions used throughout the text. Section~\ref{sec-qual} is devoted to the definition of the operator
and to the proof of the assertions (A) and (B), see Theorem~\ref{thm-basic}. The proofs are mostly based on the use of singular integral operators
previously studied in \cite{OV16} and some resolvent machineries already used in a similar (but different) context in \cite{BEHL17,BH17}.
In Section~\ref{sec-var} we deal with a more detailed study of the discrete spectrum. The key idea of the analysis
is to obtain the sesquilinear form for the square of $A_{m,\tau}$. The squared operator clearly acts as the (shifted) Laplacian
away from $\Sigma$, and the main difficulty is to understand how the transmission condition translates to $A_{m,\tau}^2$, which is settled
in Proposition~\ref{theorem_form_B_tau_square}. The approach is reminiscent of \cite[p.~379]{HMR02} and \cite{ALTR16} for other types of Dirac operators.
It turns out that the quadratic form for $A_{m,\tau}^2$ is given by the same expression
as the one for the so-called $\delta'$-potential, see e.g. \cite[Prop. 3.15]{BLL12},
but defined on a smaller domain. Hence, our construction delivers a new type of generalized surface interactions \cite{ER}.
Nevertheless, an additional geometry-induced constraint along $\Sigma$ leads to a much more involved analysis
and a completely different behavior when compared to the $\delta'$-interaction studied, e.g., in \cite{EJ14}.
In particular, Propositions~\ref{prop3.4} and~\ref{thm-disc1} cover the above points (C) and (D).
Section~\ref{sec:dis-spec} is then devoted to the study of the asymptotic regime \eqref{eq-as000},
and the points (E) and (F) follow from Corollaries~\ref{cor32} and~\ref{cor33}, which are both consequences
of a central estimate given in Theorem~\ref{thm-ev1}. In fact, the asymptotic analysis
does not use the above operator $\Upsilon_\tau$ but another unitary equivalent operator introduced in Section~\ref{ssec42}
which is easier to deal with and which implies an equivalent reformulation given in Proposition~\ref{prop-main1}.
The upper and lower bounds for the eigenvalues are then obtained separately in Subsections~\ref{sec-upp1}
and~\ref{sec-low1} respectively, by comparing the operator $A_{m,\tau}$ first with operators in thin neighborhoods
of $\Sigma$ and then, using a change of variable, with operators with separated variables in $\Sigma\times I$
with $I$ being a one-dimensional interval, whose one-dimensional part is analyzed directly similar to, e.g.,~\cite{EY,PP}.
Contrary to the approach of \cite{ALTR16} our study does not use  semi-classical type estimates,
which allows a self-contained proof.

\subsection{Notations}\label{ssec-not}

For a Hilbert space $\cH$, one denotes by $\langle\cdot,\cdot\rangle_{\cH}$ the scalar product on $\cH$ and by $\|\cdot\|_{\cH}$ the associated norm.
As there is no risk of confusion and for the sake of readability, we simply set $\|\cdot\|_{\mathbb{C}^4} = |\cdot|$ and $\|\cdot\|_{\mathbb{R}^3\otimes\mathbb{C}^4} = |\cdot|$.

By $\bfB(\cH)$ we denote the Banach space of bounded linear operators in $\cH$.
If $T$ is a self-adjoint operator in $\cH$, then we denote
by $\dom(T)$ its domain, by $\ker(T)$ and $\ran(T)$ its kernel and range respectively,
and $E_n(T)$ will stand for the $n$-th eigenvalue of $T$ when enumerated in the non-decreasing
order and counted according to multiplicities. The spectrum of $T$ is denoted by $\spec(T)$, the essential spectrum
by $\spec_\ess(T)$ and the resolvent set by $\res(T)$.
If the operator $T$ in $\cH$ is generated by a closed lower semibounded
sesquilinear form $t$ defined on the domain $\dom(t)$, then the following variational characterization
of the eigenvalues holds (min-max principle): for $n\in\NN$ set
\[
\varepsilon_n(T):=\inf_{\substack{V\subset \dom(t)\\ \dim V=n}}\sup_{\substack{u\in V\\u\ne 0}} \dfrac{t(u,u)}{\|u\|^2_{\cH}},
\]
then $E_n(T)=\varepsilon_n(T)$ if $\varepsilon_n(T)<\inf\spec_\ess(T)$, otherwise
one has $\varepsilon_m(T)=\inf\spec_\ess(T)$ for all $m\ge n$.
We sometimes write $E_n(t):=E_n(T)$ and $\varepsilon_n(t):=\varepsilon_n(T)$. Furthermore,
for $E\in\RR$ we denote by $\cN(T,E)$ the number of eigenvalues of $T$ in $(-\infty,E)$
and set $\cN(t,E):=\cN(T,E)$.

For two closed and semibounded from below sesquilinear forms $t_1$ and $t_2$ their direct sum
$t_1\oplus t_2$ is the sesquilinear form defined on $\dom(t_1\oplus t_2):=\dom(t_1) \times \dom(t_2)$
by
\[
(t_1\oplus t_2)\big((u_1,u_2),(u_1,u_2)\big):=t_1(u_1,u_1)+t_2(u_2,u_2),
\quad (u_1, u_2) \in \dom(t_1) \times \dom(t_2).
\]
If $T_1$ and $T_2$ are the operators associated with $t_1$ and $t_2$,
then the operator associated with $t_1\oplus t_2$ is $T_1\oplus T_2$, and
$\cN(t,E)=\cN(t_1,E)+\cN(t_2,E)$.
The form inequality $t_1\ge t_2$ means that $\dom(t_1)\subseteq \dom(t_2)$
and $t_1(u)\ge t_2(u)$ for all $u\in \dom(t_1)$. By the min-max principle
the form inequality implies the respective inequality for the Rayleigh quotients,
$\varepsilon_n(t_1)\ge \varepsilon_n(t_2)$ for any $n\in\NN$, and the reverse inequality for the eigenvalue counting functions,
$\cN(t_1,E)\le \cN(t_2,E)$ for all $E\in\RR$.

Let $\alpha_1$, $\alpha_2$, $\alpha_3$, $\beta$ and $\gamma_5$ be the $4\times 4$ Dirac matrices 
\begin{equation} \label{def_Dirac_matrices}
  \alpha_j := \begin{pmatrix} 0 & \sigma_j \\ \sigma_j & 0 \end{pmatrix}, 
  \quad \beta := \begin{pmatrix} I_2 & 0 \\ 0 & -I_2 \end{pmatrix}, \quad
	\gamma_5: = \begin{pmatrix} 0 & I_2 \\ I_2 & 0 \end{pmatrix},
\end{equation}
where $I_k$ denotes the $k \times k$ identity matrix and $\sigma_j$ are the $2\times2$ Pauli spin matrices,
\begin{equation*}
  \sigma_1 := \begin{pmatrix} 0 & 1 \\ 1 & 0 \end{pmatrix}, \qquad
  \sigma_2 := \begin{pmatrix} 0 & -\rmi \\ \rmi & 0 \end{pmatrix}, \qquad
  \sigma_3 := \begin{pmatrix} 1 & 0 \\ 0 & -1 \end{pmatrix}.
\end{equation*}
The Dirac matrices fulfill the anti-commutation relations
\begin{gather}\label{eq_commutation}
	\alpha_j\alpha_k + \alpha_k\alpha_j = 2\delta_{jk}I_4,\qquad
	j,k\in\{0,1,2,3\}, \quad \alpha_0 := \beta,\\
	 \label{commutation_gamma_5}
 \gamma_5 \alpha_j = \alpha_j \gamma_5, ~j \in \{ 1, 2, 3\}, \quad  \gamma_5 \beta = -\beta \gamma_5.
\end{gather}
For vectors $x = (x_1, x_2, x_3)\in\RR^3$ we employ the notation
\[
\alpha \cdot x :=  \alpha_1 x_1 +  \alpha_2 x_2 + \alpha_3 x_3,
\quad
\sigma \cdot x :=  \sigma_1 x_1 +  \sigma_2 x_2 + \sigma_3 x_3.
\]

\section{Qualitative spectral properties}\label{sec-qual}

\subsection{Definition of  the operator}

Let $\Omega_+\subset \RR^3$ be a bounded $C^4$ smooth domain. We set
\[
\Sigma:=\partial\Omega_+, \quad \Omega_-:=\RR^3\setminus \overline{\Omega_-},
\]
and denote by $\nu$ the unit normal vector field on $\Sigma$ pointing outwards of~$\Omega_+$.
For $s \in \Sigma$ and $\tau\in\RR$ we set 
\begin{equation} \label{def_M}
  \cB(s) := - \rmi \beta \alpha \cdot \nu(s), \quad \cP_\tau^\pm(s) := \frac{\tau}{2} \pm \cB(s).
\end{equation}
Note that for any  $s \in \Sigma$ the matrix $\cB(s)$ is self-adjoint and unitary by~\eqref{eq_commutation}.

For $m\in\RR$ and $\tau\in\RR$,  we denote by $A_{m,\tau}$
the operator in $L^2(\RR^3,\CC^4)\simeq L^2(\Omega_+,\CC^4)\oplus L^2(\Omega_-,\CC^4)$ acting as
\begin{equation}
    \label{equation_dom_B_tau}
\begin{aligned}
A_{m,\tau} u &= \Big((-\rmi \alpha \cdot \nabla + m \beta) u_+ , (-\rmi \alpha \cdot \nabla + m \beta) u_-\Big),\\
\dom(A_{m,\tau}) &= \big\{
			u = (u_+,u_-): u_\pm\in H^1(\Omega_\pm, \mathbb{C}^4), \ 
          \cP_\tau^- u_+ +\cP_\tau^+ u_-= 0 \text{ on } \Sigma\big\}.	
\end{aligned}
\end{equation}
For $\tau \in \mathbb{R} \setminus \{ -2,2\}$ we set
\begin{equation}
\label{def_R_tau}
  \cR_\tau^\pm := -(\cP_\tau^\mp )^{-1} \cP_\tau^\pm= \dfrac{4+\tau^2}{4-\tau^2} I_4 \pm \dfrac{4\tau}{4-\tau^2}\,\cB.
\end{equation}
Then one has the commutation relations
\begin{equation}
   \label{eqcommr1}
\cR_\tau^\pm \cB=\cB \cR_\tau^\pm, \quad
 \cR_\tau^\pm\gamma_5= \gamma_5  \cR_\tau^\mp.
\end{equation}
For $\tau \notin\{ -2, 0, 2 \}$, the transmission condition for $u\in \dom(A_{m,\tau})$ can equivalently be rewritten as
	\begin{equation}
	   \label{eqtrans2}
	u_+=\cR_\tau^+ u_- \quad \text{ or } \quad 
	u_-=\cR_\tau^- u_+ 
	\quad \text{ or } \quad
	u_+ + u_- = \dfrac{2}{\tau} \cB (u_+ - u_-).
	\end{equation}

\begin{remark}\label{rem-mit}
For $|\tau|=2$ the transmission condition in \eqref{equation_dom_B_tau} decomposes as
\begin{align*}
  u_+ &= \cB u_+, &\qquad u_- &= - \cB u_- && \text{ for }\tau = 2,\\
  u_+ &= -\cB u_+, &\qquad u_- &= \cB u_- && \text{ for } \tau = -2,
\end{align*}
i.e. $A_{m, \pm 2}$ is the orthogonal sum of Dirac operators in $\Omega_\pm$ 
with MIT bag boundary conditions as studied, e.g., in \cite{ALTR16, OV16}.
Using the language of \cite{ALTR16}, for $\tau=2$ and $m>0$ one recovers the MIT bag operator
with the positive mass $m$ in $\Omega_+$ and the one with a negative mass $(-m)$ in $\Omega_-$ (and vice versa for $\tau=-2$).

As mentioned in the introduction it was shown in \cite{MP16,MP17} that, under some technical assumptions,
the operators $A_{m,\tau}$ can be approximated by Dirac operators with regular potentials.
As $A_{m,\tau}$ approximates $A_{m,\pm 2}$ for $\tau$ tending to $\pm2$, this could provide a new
interpretation and regularization of MIT bag operators with negative masses,
namely as the restriction of the limit of Dirac operators with suitable squeezed potentials and positive mass.
The missing point in this program is the fact that the technical restrictions of \cite{MP16}
do not allow to study the values of $\tau$ close to $\pm 2$.
\end{remark}

\begin{remark} The transmission condition in \eqref{equation_dom_B_tau} corresponds to the operator acting as formally written in \eqref{eqn:op_form}, cf.~\cite[Section~5]{AMV15}.
Indeed, for $u =(u_+,u_-)\in H^1(\Omega_+,\mathbb{C}^4)\times H^1(\Omega_-,\mathbb{C}^4)$ let us define the distribution $\delta_\Sigma u$ by its action
$$
\langle \delta_\Sigma u,\varphi\rangle = \frac12\iint_{\Sigma}(u_+ + u_-)\varphi \,\dd\Sigma,\quad\varphi\in C_0^\infty(\mathbb{R}^3,\mathbb{C}^4).
$$
with $\dd\Sigma$ being the surface measure. When computing $A_{m,\tau} u$ in the distributional sense using the above definition of $\delta_\Sigma u$
and the expression given in \eqref{eqn:op_form}, one sees that the transmission condition in \eqref{equation_dom_B_tau} ensures that $A_{m,\tau} u$
belongs to $L^2(\mathbb{R}^3,\mathbb{C}^4)$.
\end{remark}

Let us list some basic properties of the operator $A_{m, \tau}$:

\begin{thm}\label{thm-basic}
The operator $A_{m,\tau}$ defined in~\eqref{equation_dom_B_tau} is self-adjoint, and the following assertions hold true:
\begin{enumerate}
\item[(a)] the essential spectrum of $A_{m,\tau}$ is $\big(-\infty,-|m|\big]\mathop{\cup}\big[|m|,+\infty)$,
\item[(b)] the spectrum of $A_{m,\tau}$  is symmetric with respect to $0$,
\item[(c)] each eigenvalue of $A_{m,\tau}$ has an even multiplicity,
\item[(d)] for $\tau\neq0$, the operator $A_{m,\tau}$ is unitarily equivalent to $A_{m,\frac{4}{\tau}}$,
\item[(e)] the operator $A_{-m,-\tau}$ is unitarily equivalent to $A_{m,\tau}$.
\end{enumerate}
\end{thm}

The results will be deduced from~\cite{BEHL17, BH17} by applying the abstract machinery
developed there for suitable boundary conditions. To keep the paper self-contained
we give a complete  proof in the rest of this section. We first introduce some
related integral operators in Section~\ref{section_free_Operator},
and with their help we prove the self-adjointness of $A_{m,\tau}$
in Proposition~\ref{theorem_basic_properties}.
The points (a)--(e) are justified in Section~\ref{sec-ae}.

\subsection{Auxiliary integral operators} \label{section_free_Operator}

First, we define the free Dirac operator and discuss some of its properties which will be needed 
for our further considerations. Recall the definition of the Dirac matrices $\alpha_j$ and $\beta$ 
from~\eqref{def_Dirac_matrices}. Then, the free Dirac operator $A_{m, 0}$ is given by
\begin{equation} \label{def_free_Dirac}
  A_{m, 0} u := -\rmi \sum_{j=1}^3 \alpha_j \partial_j u + m \beta u = -\rmi (\alpha \cdot \nabla) u + m \beta u, 
  \quad \dom (A_{m, 0}) = H^1(\mathbb{R}^3, \mathbb{C}^4).
\end{equation}
With the help of the Fourier transform one easily sees that $A_{m, 0}$ is self-adjoint and that
\begin{equation} \label{spectrum_A_{m, 0}}
  \spec(A_{m, 0}) = \spec_\ess(A_{m, 0}) = (-\infty, -|m|] \cup [|m|, \infty).
\end{equation}
For $\lambda \in \res(A_{m, 0}) = \mathbb{C} \setminus \big( (-\infty, -|m|] \cup [|m|, \infty) \big)$
the resolvent of $A_{m, 0}$ is given by
\begin{gather*}
  (A_{m, 0} - \lambda)^{-1} u(x) = \iiint_{\mathbb{R}^3} G_\lambda(x-y) u(y) \dd y,\\
  G_\lambda(x) = \left( \lambda I_4 + m \beta 
     + \left( 1 - \rmi \sqrt{\lambda^2 - m^2} |x| \right) \frac{\rmi(\alpha \cdot x )}{|x|^2} \right)
    \frac{e^{\rmi \sqrt{\lambda^2 - m^2} |x|}}{4 \pi |x|};
\end{gather*}
cf. \cite[Section~1.E]{T92} or \cite[Lemma~2.1]{AMV14}. 
In this formula we use the convention $\Im \sqrt{\lambda^2 - m^2} > 0$.
The resolvent of $A_{m, 0}$ and the particular form 
of its integral kernel will be important later for the 
basic spectral analysis of the Dirac operator with a Lorentz scalar $\delta$-shell interaction.

Now we are going to discuss some integral operators which are related to the Green's function~$G_\lambda$.
For~$\lambda \in \res(A_{m, 0})$
we define $\Phi_\lambda: L^2(\Sigma, \mathbb{C}^4) \rightarrow L^2(\mathbb{R}^3, \mathbb{C}^4)$ acting as
\begin{equation} \label{def_gamma_lambda}
  \Phi_\lambda \varphi(x) := \iint_\Sigma G_\lambda(x-y) \varphi(y) \dd\Sigma(y), 
  \quad \varphi \in L^2(\Sigma, \mathbb{C}^4),~x \in \mathbb{R}^3,
\end{equation}
and $\mathcal{C}_\lambda: H^{\frac{1}{2}}(\Sigma, \mathbb{C}^4) \rightarrow H^{\frac{1}{2}}(\Sigma, \mathbb{C}^4)$,
\begin{equation} \label{def_M_lambda}
  \mathcal{C}_\lambda \varphi(x) := \lim_{\varepsilon \searrow 0} \iint_{\Sigma \setminus B(x, \varepsilon)} 
  G_\lambda(x-y) \varphi(y) \dd\Sigma(y), \quad
  \varphi \in H^{\frac{1}{2}}(\Sigma, \mathbb{C}^4),~x \in \Sigma,
\end{equation}
where $\dd\Sigma$ is the surface measure on $\Sigma$ and $B(x,\varepsilon)$ is the ball of radius $\varepsilon$ centered at $x$.
Both operators $\Phi_\lambda$ and $\mathcal{C}_\lambda$ are well-defined and bounded, 
see~\cite[Proposition~3.4]{BEHL17} and~\cite[Proposition~4.2~(ii)]{BH17} or \cite[Sections~2.1 and~2.2]{OV16}, 
and $\Phi_\lambda$ is injective by~\cite[Proposition~3.4 and Definition~2.3]{BEHL17}. We also note the useful property
\begin{equation} \label{equation_gamma_smooth}
  \Phi_\lambda \varphi \in H^1(\Omega_+, \mathbb{C}^4) \oplus H^1(\Omega_-, \mathbb{C}^4)
  \quad \text{for} \quad \varphi \in H^{\frac{1}{2}}(\Sigma, \mathbb{C}^4);
\end{equation}
cf.~\cite[Proposition~4.2~(i)]{BH17}.
Moreover, if $\lambda \in \res(A_{m, 0})$, then a function $u_\lambda \in H^1(\Omega_+, \mathbb {C}^4) \oplus H^1(\Omega_-, \mathbb{C}^4)$
satisfies 
\begin{equation*}
  (-\rmi \alpha \cdot \nabla + m \beta - \lambda) u_\lambda = 0 \text{ in } \Omega_\pm
\end{equation*}
iff there exists $\varphi \in H^{\frac{1}{2}}(\Sigma, \mathbb{C}^4)$ such that 
\begin{equation} \label{equation_kernel}
  u_\lambda = \Phi_\lambda \varphi;
\end{equation}
cf.~\cite[Proposition~4.2]{BH17}. 
The adjoint $\Phi_\lambda^*: L^2(\mathbb{R}^3, \mathbb{C}^4) \rightarrow L^2(\Sigma, \mathbb{C}^4)$ of $\Phi_\lambda$
acts as 
\begin{equation} \label{equation_gamma_lambda_star}
  \Phi_\lambda^* u = \big( (A_{m, 0} - \overline{\lambda})^{-1} u \big)\big|_\Sigma
\end{equation}
and it has the more explicit representation
\begin{equation*}
  \Phi_\lambda^* u(x) = \iiint_{\mathbb{R}^3} G_{\overline{\lambda}}(x-y) u(y) \dd y, \quad
  u \in L^2(\mathbb{R}^3, \mathbb{C}^4),~x \in \Sigma.
\end{equation*}

Let $\varphi \in H^{\frac{1}{2}}(\Sigma, \mathbb{C}^4)$ and $\lambda \in \res(A_{m, 0})$. Then, the trace on $\Sigma$ of 
\begin{equation*}
  \Phi_\lambda \varphi = \big( (\Phi_\lambda \varphi)_+, (\Phi_\lambda \varphi)_- \big) 
  \in H^1(\Omega_+, \mathbb{C}^4) \oplus H^1(\Omega_-, \mathbb{C}^4)
\end{equation*}
is $\big((\Phi_\lambda \varphi)_\pm\big) \big|_\Sigma = \mathcal{C}_\lambda \varphi \mp \dfrac{\rmi}{2} (\alpha \cdot \nu) \varphi$,
see~\cite[Lemma~2.2]{AMV15} for $\lambda \in (-|m|, |m|)$; the case $\lambda \in \mathbb{C} \setminus \mathbb{R}$
can be shown in the same way. In particular, we have
\begin{gather} \label{jump1}
  \frac{1}{2} \left( (\Phi_\lambda \varphi)_+ + (\Phi_\lambda \varphi)_- \right)
  = \mathcal{C}_\lambda \varphi \text{ on } \Sigma,\\
 \label{jump2}
  \rmi \alpha \cdot \nu \left( (\Phi_\lambda \varphi)_+ - (\Phi_\lambda \varphi)_- \right) = \varphi
  \text{ on } \Sigma.
\end{gather}
The operator $\mathcal{C}_\lambda^2 - \frac{1}{4}I_4$ can be extended to a bounded operator
\begin{equation} \label{C_lambda_compact}
  \widetilde{\mathcal{C}}_\lambda^2 - \frac{1}{4} I_4: H^{-\frac{1}{2}}(\Sigma, \mathbb{C}^4) \rightarrow H^{\frac{1}{2}}(\Sigma, \mathbb{C}^4),
\end{equation}
see~\cite[Proposition~4.4~(iii)]{BH17} and also~\cite[Proposition~2.8]{OV16}.
In particular, the operator $\big(\mathcal{C}_\lambda^2 - \frac{1}{4}I_4 \big)$ is compact 
in $H^{\frac{1}{2}}(\Sigma, \mathbb{C}^4)$.

We end this section with a variant of the Birman-Schwinger principle for the operator $A_{m, \tau}$.
It is a special variant of the general result stated in~\cite[Theorem~2.4]{BEHL17} 
or~\cite[Proposition~3.1]{AMV15}; to keep the presentation self-contained,
we add a short and simple proof here.

\begin{lem} \label{lemma_Birman_Schwinger}
  Let $A_{m, \tau}$ be defined as in~\eqref{equation_dom_B_tau}
  and let $\tau \in \mathbb{R}$. Then $\lambda \in \res(A_{m, 0})$
  is an eigenvalue of $A_{m, \tau}$ if and only if $-1$ is an eigenvalue of $\tau \beta \mathcal{C}_\lambda$.
\end{lem}
\begin{proof}
  Assume that $\lambda \in \res(A_{m, 0})$ is an eigenvalue of $A_{m, \tau}$ with eigenfunction $u_\lambda$.
  Then, by~\eqref{equation_kernel}
  there exists a $0 \neq \varphi \in H^{\frac{1}{2}}(\Sigma, \mathbb{C}^4)$ such that 
  $u_\lambda = \Phi_\lambda \varphi$. Since $u_\lambda \in \dom(A_{m, \tau})$ it holds
  $\mathcal{P}_\tau^- u_{\lambda, +} + \mathcal{P}_\tau^+ u_{\lambda, -} = 0$.
  Using the definitions of the matrices $\mathcal{P}_\tau^\pm$ from \eqref{def_M} and~\eqref{jump1} and \eqref{jump2} this yields
  \begin{equation*}
    \begin{split}
      0 &= \mathcal{P}_\tau^- (\Phi_\lambda \varphi)_+ + \mathcal{P}_\tau^+ (\Phi_\lambda \varphi)_-\\
      &= \rmi \beta \alpha \cdot \nu \big( (\Phi_\lambda \varphi)_+ - (\Phi_\lambda \varphi)_-\big) 
          + \frac{\tau}{2} \big( (\Phi_\lambda \varphi)_+ + (\Phi_\lambda \varphi)_-\big) \\
      &= \beta (I_4 + \tau \beta \mathcal{C}_\lambda) \varphi,
    \end{split}
  \end{equation*}
  i.e. $-1$ is an eigenvalue of $\tau \beta \mathcal{C}_\lambda$.

  Conversely, if $-1$ is an eigenvalue of $\tau \beta \mathcal{C}_\lambda$ with non-trivial eigenfunction $\varphi$, 
  then $u_\lambda := \Phi_\lambda \varphi \neq 0$ satisfies
  $u_\lambda \in H^1(\Omega_+, \mathbb{C}^4) \oplus H^1(\Omega_-, \mathbb{C}^4)$ 
  by~\eqref{equation_gamma_smooth}. Moreover, employing again~\eqref{jump1} and \eqref{jump2} we obtain
  \begin{equation*}
    \begin{split}
      \mathcal{P}_\tau^- u_{\lambda, +} + \mathcal{P}_\tau^+ u_{\lambda, -} 
          &= \mathcal{P}_\tau^- (\Phi_\lambda \varphi)_+ + \mathcal{P}_\tau^+ (\Phi_\lambda \varphi)_-\\
      &= \rmi \beta \alpha \cdot \nu \big( (\Phi_\lambda \varphi)_+ - (\Phi_\lambda \varphi)_-\big)
          + \frac{\tau}{2} \big( (\Phi_\lambda \varphi)_+ + (\Phi_\lambda \varphi)_-\big) \\
      &= \beta (I_4 + \tau \beta \mathcal{C}_\lambda) \varphi = 0,
    \end{split}
  \end{equation*}
  as $\varphi \in \ker(I_4 + \tau \beta \mathcal{C}_\lambda)$. This shows $u_\lambda \in \dom(A_{m, \tau})$.
  Eventually, equation~\eqref{equation_kernel} implies 
  \begin{equation*}
    (A_{m, \tau} - \lambda) u_\lambda = (A_{m, \tau} - \lambda) \Phi_\lambda \varphi = 0
  \end{equation*}
  and hence $\lambda$ is an eigenvalue of $A_{m, \tau}$.
\end{proof}

Using Lemma~\ref{lemma_Birman_Schwinger} and a result from~\cite{AMV15} we deduce finally, that 
$A_{m, \tau}$ has no eigenvalues in $(-|m|, |m|)$, if the interaction strength $\tau$ is small.

\begin{cor} \label{corollary_no_discrete_spectrum_in_gap}
  There exists $\tau_m > 0$ such that $A_{m, \tau}$ has no eigenvalues in $(-|m|, |m|)$ for all $|\tau| < \tau_m$.
\end{cor}
\begin{proof}
  First, by \cite[Lemma~3.2]{AMV15} there exists a constant $C(m):=C > 0$ independent of $\lambda$ such that
  \begin{equation*}
    \| \mathcal{C}_\lambda \varphi \|_{L^2(\Sigma, \mathbb{C}^4)}
        \leq C \| \varphi \|_{L^2(\Sigma, \mathbb{C}^4)}
        \qquad \forall \varphi \in H^{\frac{1}{2}}(\Sigma, \mathbb{C}^4), \lambda \in (-|m|, |m|).
  \end{equation*}
  Hence, if $\tau < \tau_m := \frac{1}{C}$, then $-1$ can not be an eigenvalue of
  $\tau \beta \mathcal{C}_\lambda$. From Lemma~\ref{lemma_Birman_Schwinger} we conclude that
  $A_{m, \tau}$ can not have eigenvalues in $(-|m|, |m|)$ for $\tau < \tau_m$.
\end{proof}

\subsection{Proof of self-adjointness}

First, we prove that $A_{m, \tau}$ is symmetric:

\begin{lem} \label{lemma_symmetric}
  Let $m, \tau \in \mathbb{R}$, then the operator $A_{m, \tau}$ given by~\eqref{equation_dom_B_tau}  is symmetric.
\end{lem}
\begin{proof}
  Let $u \in \dom(A_{m, \tau})$. Employing an integration by parts we have
  \begin{multline*}
      \langle A_{m, \tau} u, u \rangle_{L^2(\mathbb{R}^3, \mathbb{C}^4)}
              - \langle  u, A_{m, \tau} u \rangle_{L^2(\mathbb{R}^3, \mathbb{C}^4)} = \langle -\rmi \alpha \cdot \nu u_+, u_+ \rangle_{L^2(\Sigma, \mathbb{C}^4)}           -\langle -\rmi \alpha \cdot \nu u_-, u_- \rangle_{L^2(\Sigma, \mathbb{C}^4)} \\
      = \frac{1}{2} \langle -i \alpha \cdot \nu (u_+ - u_-), u_+ + u_- \rangle_{L^2(\Sigma, \mathbb{C}^4)} 
      - \frac{1}{2} \langle u_+ + u_-, -i \alpha \cdot \nu (u_+ - u_-) \rangle_{L^2(\Sigma, \mathbb{C}^4)}.
  \end{multline*}
  Using the transmission condition~\eqref{eqtrans2}, the anti-commutation relation~\eqref{eq_commutation}
  and $\beta^2 = I_4$ the last term can be rewritten
  \begin{multline*}
      \frac{1}{2} \langle -i \alpha \cdot \nu (u_+ - u_-), u_+ + u_- \rangle_{L^2(\Sigma, \mathbb{C}^4)} 
          - \frac{1}{2} \langle u_+ + u_-, -i \alpha \cdot \nu (u_+ - u_-) \rangle_{L^2(\Sigma, \mathbb{C}^4)} \\
      = \frac{\tau}{4} \langle \beta (u_+ + u_-), u_+ + u_- \rangle_{L^2(\Sigma, \mathbb{C}^4)} 
          - \frac{\tau}{4} \langle u_+ + u_-, \beta (u_+ + u_-) \rangle_{L^2(\Sigma, \mathbb{C}^4)} = 0,
  \end{multline*}
  which shows that $\langle A_{m, \tau} u, u \rangle_{L^2(\mathbb{R}^3, \mathbb{C}^4)} \in \mathbb{R}$. Since 
  $u \in \dom(A_{m, \tau})$ was arbitrary, the claim of this lemma follows.
\end{proof}

The following technical result will play a crucial role in the proof of the self-adjointness of $A_{m, \tau}$:

\begin{lem} \label{lemma_inverse_M}
  Let~$\tau \in \mathbb{R}$ and let for $\lambda \in \mathbb{C} \setminus \mathbb{R}$ the operator 
  $\mathcal{C}_\lambda$ be defined by~\eqref{def_M_lambda}. Then the operator $I_4 + \tau \beta \mathcal{C}_\lambda$
   admits a bounded and everywhere defined inverse in $H^{\frac{1}{2}}(\Sigma, \mathbb{C}^4)$.
\end{lem}
\begin{proof}
  First, we note that $I_4 + \tau \beta \mathcal{C}_\lambda$ is injective, as otherwise the symmetric operator 
  $A_{m, \tau}$ would have the non-real eigenvalue $\lambda$ by Lemma~\ref{lemma_Birman_Schwinger}. To show that
  $I_4 + \tau \beta \mathcal{C}_\lambda$ is also surjective note
  \begin{equation*}
    \begin{split}
      \ran\big( I_4 + \tau \beta \mathcal{C}_\lambda \big) 
          &\supset \ran\big[ \big( I_4 + \tau \beta \mathcal{C}_\lambda \big)
          \big( I_4 - \tau \beta \mathcal{C}_\lambda \big) \big] \\
      &= \ran\big[ \big( I_4 + \tau \beta \mathcal{C}_\lambda \big)
          \big( I_4 + \mathcal{C}_\lambda \tau \beta - \tau ( \mathcal{C}_\lambda \beta + \beta \mathcal{C}_\lambda) \big) \big]\\
      &= \ran \big[ I_4 +\tau^2 \beta \mathcal{C}_\lambda^2 \beta 
          - \tau^2 \beta \mathcal{C}_\lambda ( \beta \mathcal{C}_\lambda + \mathcal{C}_\lambda \beta )\big].
    \end{split}
  \end{equation*}
  Using the anti-commutation relations~\eqref{eq_commutation} we obtain that~$\beta \mathcal{C}_\lambda + \mathcal{C}_\lambda \beta$
  is an integral operator with kernel
  \begin{equation*}
    K(x, y) = \big( \lambda \beta + m I_4 \big) \frac{e^{\rmi \sqrt{\lambda^2 - m^2} |x-y|}}{2 \pi |x-y|},
  \end{equation*}
  i.e. $\beta \mathcal{C}_\lambda + \mathcal{C}_\lambda \beta$ is a constant matrix times the 
  single layer boundary integral operator for $-\Delta + m^2 - \lambda^2$ 
  which is compact in $H^{\frac{1}{2}}(\Sigma, \mathbb{C}^4)$; cf., e.g., \cite[Theorem~6.11]{M00}.
  Moreover, by~\eqref{C_lambda_compact} 
  also $\mathcal{C}_\lambda^2 - \frac{1}{4} I_4$ is compact in~$H^{\frac{1}{2}}(\Sigma, \mathbb{C}^4)$.
  Since $\mathcal{C}_\lambda$ is bounded in $H^{\frac{1}{2}}(\Sigma, \mathbb{C}^4)$ we deduce that
  \begin{equation*}
    \mathcal{K}_\lambda := \tau^2 \beta \left( \mathcal{C}_\lambda^2 - \frac{1}{4} I_4 \right) \beta 
        - \tau^2 \beta \mathcal{C}_\lambda ( \beta \mathcal{C}_\lambda + \mathcal{C}_\lambda \beta)
  \end{equation*}
  is compact in $H^{\frac{1}{2}}(\Sigma, \mathbb{C}^4)$.
  Note that both operators $I_4 + \tau \beta \mathcal{C}_\lambda$ and $I_4 - \tau \beta \mathcal{C}_\lambda$ are injective,
  as otherwise one of the symmetric operators $A_{m, \pm \tau}$ would have the non-real eigenvalue $\lambda$ by 
  Lemma~\ref{lemma_Birman_Schwinger}.
  Hence, we get finally by Fredholm's alternative that
  \begin{equation*}
     \ran \big[ I_4 + \tau^2 \beta \mathcal{C}_\lambda^2 \beta 
          - \tau^2 \beta \mathcal{C}_\lambda ( \beta \mathcal{C}_\lambda + \mathcal{C}_\lambda \beta) \big] 
      = \ran \left[ \left(1 + \frac{\tau^2}{4} \right) I_4 + \mathcal{K}_\lambda \right]
          = H^{\frac{1}{2}}(\Sigma, \mathbb{C}^4).
     \end{equation*}
  Therefore, we deduce eventually 
  \begin{equation*}
    \begin{split}
      \ran\big( I_4 + \tau \beta \mathcal{C}_\lambda \big) 
          &\supset \ran \big[ I_4 + \tau^2 \beta \mathcal{C}_\lambda^2 \beta 
          - \tau^2 \beta \mathcal{C}_\lambda ( \beta \mathcal{C}_\lambda + \mathcal{C}_\lambda \beta) \big] 
          = H^{\frac{1}{2}}(\Sigma, \mathbb{C}^4) 
    \end{split}
  \end{equation*}
  and thus, $I_4 + \tau \beta \mathcal{C}_\lambda$ is surjective.
  This shows that the closed operator $I_4 + \tau \beta \mathcal{C}_\lambda$ is bijective and hence,
  it admits a bounded and everywhere defined inverse by the closed graph theorem.
\end{proof}

Now, we are prepared to prove the self-adjointness of $A_{m, \tau}$ which is the central point of Theorem~\ref{thm-basic}:

\begin{prop} \label{theorem_basic_properties}
  Let $m, \tau \in \mathbb{R}$ and let $A_{m, \tau}$ be defined by \eqref{equation_dom_B_tau}. 
  Then, $A_{m, \tau}$ is self-adjoint and for any $\lambda \in \mathbb{C} \setminus \mathbb{R}$ one has the resolvent formula
  \begin{equation*}
    (A_{m, \tau} - \lambda)^{-1} = (A_{m, 0} - \lambda)^{-1} - \Phi_\lambda \big( I_4 + \tau \beta \mathcal{C}_\lambda \big)^{-1} 
        \tau \beta \Phi_{\overline{\lambda}}^*.
  \end{equation*}
\end{prop}
\begin{proof}
  Since $A_{m, \tau}$ is symmetric by Lemma~\ref{lemma_symmetric} it suffices to show that
  $\ran(A_{m, \tau} - \lambda) = L^2(\mathbb{R}^3, \mathbb{C}^4)$ for $\lambda \in \mathbb{C} \setminus \mathbb{R}$.
  Let $\lambda \in \mathbb{C} \setminus \mathbb{R}$ and $v \in L^2(\mathbb{R}^3, \mathbb{C}^4)$ be fixed.
  We define
  \begin{equation} \label{Krein1}
    u := (A_{m, 0} - \lambda)^{-1} v - \Phi_\lambda \big( I_4 + \tau \beta \mathcal{C}_\lambda \big)^{-1} 
        \tau \beta \Phi_{\overline{\lambda}}^* \, v.
  \end{equation}
  Note that $u$ is well-defined, as 
  $\Phi_{\overline{\lambda}}^* \, v = \big( (A_{m, 0} - \lambda)^{-1} v \big)  \big|_\Sigma \in H^{\frac{1}{2}}(\Sigma, \mathbb{C}^4)$,
  see \eqref{equation_gamma_lambda_star},
  and $I_4 + \tau \beta \mathcal{C}_\lambda$ is bijective in $H^{\frac{1}{2}}(\Sigma, \mathbb{C}^4)$ by Lemma~\ref{lemma_inverse_M}.
  We are going to prove that $u \in \dom(A_{m, \tau})$ and $(A_{m, \tau} - \lambda) u = v$. Then, this implies the claim on the 
  range of $A_{m, \tau} - \lambda$ and the resolvent formula.
  
  Due to the mapping properties of $\Phi_{\overline{\lambda}}^*$ and $\mathcal{C}_\lambda$ we have 
  \begin{equation*}
    \big( I_4 + \tau \beta \mathcal{C}_\lambda \big)^{-1} 
        \tau \beta \Phi_{\overline{\lambda}}^* \, v \in H^{\frac{1}{2}}(\Sigma, \mathbb{C}^4).
  \end{equation*}
  Therefore, we have by~\eqref{equation_gamma_smooth} that $u \in H^1(\Omega_+, \mathbb{C}^4) \oplus H^1(\Omega_-, \mathbb{C}^4)$.
  Moreover, using~\eqref{jump1}, \eqref{jump2}, \eqref{equation_gamma_lambda_star} 
  and $\dom (A_{m, 0}) = \ran(A_{m, 0} - \lambda) = H^1(\mathbb{R}^3, \mathbb{C}^4)$ 
  we deduce
  \begin{equation*}
    \begin{split}
      \mathcal{P}_\tau^-& u_+ + \mathcal{P}_\tau^+ u_- 
          = \frac{\tau}{2} (u_+ + u_-) + \rmi \beta \alpha \cdot \nu (u_+ - u_-) \\
      &= \tau \big( (A_{m, 0} - \lambda)^{-1} v \big) \big|_\Sigma 
          - \tau \mathcal{C}_\lambda \big( I_4 + \tau \beta \mathcal{C}_\lambda \big)^{-1} \tau \beta \Phi_{\overline{\lambda}}^*\, v
          - \beta \big( I_4 + \tau \beta \mathcal{C}_\lambda \big)^{-1} \tau \beta \Phi_{\overline{\lambda}}^*\, v \\
      &= \tau \Phi_{\overline{\lambda}}^* \,v - \beta (I_4 + \tau \beta \mathcal{C}_\lambda)
          \big( I_4 + \tau \beta \mathcal{C}_\lambda \big)^{-1} \tau \beta \Phi_{\overline{\lambda}}^*\, v = 0,
    \end{split}
  \end{equation*}
  i.e. $u \in \dom(A_{m, \tau})$. Using \eqref{equation_kernel} we have 
  $(A_{m, \tau} - \lambda) u = v$. Hence, the theorem is shown.
\end{proof}

\subsection{Basic properties}\label{sec-ae}

In this section we are going to prove the points  (a)--(e) of Theorem \ref{thm-basic}.
To prove (a) take any $\lambda \in \mathbb{C} \setminus \mathbb{R}$.
First, we note that by Lemma~\ref{lemma_inverse_M} the inverse
$(I_4 + \tau \beta \mathcal{C}_\lambda)^{-1}$ is a bounded operator in $H^{\frac{1}{2}}(\Sigma, \mathbb{C}^4)$.
Moreover, since $\ran \Phi_{\overline{\lambda}}^* = H^{\frac{1}{2}}(\Sigma, \mathbb{C}^4)$, see~\eqref{equation_gamma_lambda_star}, 
and since $\Phi_{\overline{\lambda}}^*: L^2(\mathbb{R}^3, \mathbb{C}^4) \rightarrow L^2(\Sigma, \mathbb{C}^4)$ is bounded,
it follows from the closed graph theorem that the product
\begin{equation*}
  \big( I_4 + \tau \beta \mathcal{C}_\lambda \big)^{-1} \tau \beta \Phi_{\overline{\lambda}}^*:
  L^2(\mathbb{R}^3, \mathbb{C}^4) \rightarrow H^{\frac{1}{2}}(\Sigma, \mathbb{C}^4)
\end{equation*}
is bounded. As the embedding $\iota: H^{\frac{1}{2}}(\Sigma, \mathbb{C}^4) \rightarrow L^2(\Sigma, \mathbb{C}^4)$
is compact and $\Phi_\lambda$ is bounded, we deduce with the help of Theorem~\ref{theorem_basic_properties} that
\begin{equation*}
  (A_{m, \tau} - \lambda)^{-1} - (A_{m, 0} - \lambda)^{-1}  
  = -\Phi_\lambda \big( I_4 + \tau \beta \mathcal{C}_\lambda \big)^{-1} \tau \beta \Phi_{\overline{\lambda}}^*
\end{equation*}
is compact in $L^2(\mathbb{R}^3, \mathbb{C}^4)$. 
Hence, we find 
\begin{equation*}
  \spec_\ess(A_{m, \tau}) = \spec_\ess(A_{m, 0}) = (-\infty, -|m|] \cup [|m|, \infty).
\end{equation*}
This is statement~(a) of Theorem~\ref{thm-basic}.

Next, we define the charge conjugation operator $\cC$ and the time reversal operator $\cT$ by
$\cC u := \rmi \beta \alpha_2 \overline{u}$ and $\cT u := -\rmi \gamma_5 \alpha_2 \overline{u}$.
Then a simple computation shows that $\cC^2 = -\cT^2 = \mathrm{Id}$. Furthermore,
$\mathcal{C}$ and $\mathcal{T}$ leave $\dom(A_{m, \tau})$ invariant and
      \begin{equation*}
     A_{m,\tau} \cC = -\cC A_{m,\tau} \qquad \text{ and } \qquad A_{m,\tau} \cT = \cT A_{m,\tau}. 
    \end{equation*}

Assume that $\lambda\in \spec (A_{m,\tau})$. Then there exists a sequence $(u_j)\in\dom(A_{m,\tau})$ with $\|u_j\|=1$
and $(A_{m,\tau}-\lambda)u_j\to 0$ when $j\rightarrow+\infty$. Then for $v_j:=\cC u_j$ one has $\|v_j\|=1$ and 
$(A_{m,\tau}+\lambda)v_j = -\cC(A_{m,\tau}-\lambda)u_j\to 0$, i.e. $-\lambda\in \spec (A_{m,\tau})$.
This proves the point (b).

Furthermore, if $u\in \ker(A_{m,\tau}-\lambda)$, then also $\cT u\in \ker(A_{m,\tau}-\lambda)$. Moreover $\cT^2 u=-u$
and a simple calculation using the definition of~$\cT$ 
shows $\langle u,\cT u\rangle_{L^2(\RR^3,\CC^4)}= - \langle u,\cT u\rangle_{L^2(\RR^3,\CC^4)}$ 
and hence $\langle u,\cT u\rangle_{L^2(\RR^3,\CC^4)}=0$. This proves (c).

In order to prove (d), we note first that the claim is trivial for $\tau = \pm 2$. For $\tau \neq \pm 2$ consider the unitary transform
$$
	V :\left\{\begin{array}{ccc} L^2(\Omega_+,\mathbb{C}^4) \oplus L^2(\Omega_-,\mathbb{C}^4) &\longrightarrow& L^2(\Omega_+,\mathbb{C}^4) \oplus L^2(\Omega_-,\mathbb{C}^4)\\
	(u_+,u_-) & \mapsto & (u_+,-u_-).
		\end{array}\right.
$$
Let $\tau \neq 0$. For $u\in\dom(A_{m,\tau})$ we have $(Vu) \in \dom(A_{m,\frac4\tau})$ because $\cR_\tau^+ = - \cR_{\frac4\tau}^+$. Hence, we have $	A_{m,\tau} = V^{-1} A_{m,\frac4\tau} V$
which yields that $A_{m,\tau}$ and $A_{m,\frac4\tau}$ are unitarily equivalent.

Finally, the point (e) follows from the pointwise equality $\gamma_5 A_{m,\tau}= A_{-m,-\tau}\gamma_5$.

\section{Variational approach}\label{sec-var}

\subsection{Quadratic form for the square of the operator}

In order to proceed with a more detailed study, let us introduce some geometric quantities.
Throughout this section assume that $\Sigma$ be the boundary of a bounded $C^4$ smooth domain.
Recall that at each point $s\in\Sigma$ the Weingarten map $S:T_s\Sigma\to T_s\Sigma$ is defined by
$S:=\dd \nu(s)$. Its eigenvalues $\kappa_1$ and $\kappa_2$  are called the principal curvatures,
and we denote by
\begin{equation} \label{def_curvature}
M:=\dfrac{\kappa_1+\kappa_2}{2}, \quad K:=\kappa_1 \kappa_2,
\end{equation}
the mean  curvature and the Gauss curvature of $\Sigma$, respectively.

The following result will be of crucial importance for the subsequent analysis:

\begin{prop} \label{theorem_form_B_tau_square}
  Let $m\in\RR$ and $\tau \in \RR\setminus\{0,-2,2\}$, then for any $u\in \dom(A_{m,\tau})$ there holds
	\begin{multline} \label{equation_form}
          \| A_{m,\tau} u \|^2_{L^2(\RR^3,\CC^4)} = \iiint_{\RR^3\setminus \Sigma} \big|\nabla u\big|^2\, \dd x + m^2 \iiint_{\RR^3} | u|^2\, \dd x\\
          + \frac{2 m}{\tau} \iint_\Sigma | u_+ - u_-|^2 \dd\Sigma + \iint_\Sigma M |u_+|^2 \dd\Sigma - \iint_\Sigma M |u_-|^2 \dd\Sigma
  \end{multline}
	with $\dd\Sigma$ being the surface measure on $\Sigma$.
 \end{prop}

The proof of Proposition~\ref{theorem_form_B_tau_square} will use a couple of preliminary computations. 
First recall the elementary equality
\begin{equation}
\label{lemma_alpha_nu}
  (\alpha \cdot \nu) \cdot (\alpha \cdot \nabla) - \nu \cdot \nabla I_4
  = \rmi \gamma_5 \alpha \cdot (\nu \times \nabla).
\end{equation}
Other important identities are summarized in the following lemma.
Recall that for two operators $A$ and $B$ one denotes by $[A,B]:=AB-BA$ their commutator.
\begin{lem} \label{lemma_intergation_by_parts}
  Let $\Omega \subset \mathbb{R}^3$ be an open set with compact $C^4$ smooth boundary,
  let $\nu$ be the outward pointing normal vector field on the boundary, let $M$ be the mean curvature on $\partial \Omega$,
  and let $\cB$ be defined by~\eqref{def_M}.
  Then for $u \in H^2(\Omega,\CC^4)$ the following identities hold:
\begin{gather} \label{lemma_commutation_diff_B}
      \big[\alpha \cdot (\nu\times \nabla), \cB \big] u = -2 \rmi M \gamma_5 \cB  u \text{ on } \partial\Omega,\\
   \label{equation_alpha_term}
    \| \alpha \cdot \nabla u \|_{L^2(\Omega,\CC^4)}^2 = \| \nabla u \|^2_{L^2(\Omega,\RR^3\otimes\CC^4)}
      + \big\langle u, \rmi \gamma_5 \alpha \cdot (\nu \times \nabla) u \big\rangle_{L^2(\partial \Omega,\CC^4)},\\
		\label{equation_beta_term1} 2 \Re \langle -\rmi \alpha \cdot \nabla u, \beta u\rangle_{L^2(\Omega,\CC^4)} = \langle -\rmi \alpha \cdot \nu u, \beta u\rangle_{L^2(\partial\Omega,\CC^4)}.
	\end{gather}
In particular,
\begin{multline}
   \label{eqpart1}
		\| (-\rmi \alpha \cdot \nabla + m \beta) u \|_{L^2(\Omega,\CC^4)}^2   = \| \nabla u \|_{L^2(\Omega,\RR^3\otimes\CC^4)}^2 + m^2 \| u \|_{L^2(\Omega,\CC^4)}^2\\
		+ \langle -\rmi \alpha \cdot \nu u, m \beta u \rangle_{L^2(\partial\Omega,\CC^4)} 
		+ \big\langle u, \rmi \gamma_5 \alpha \cdot (\nu \times \nabla) u \big\rangle_{L^2(\partial\Omega,\CC^4)}. 
\end{multline}
\end{lem}

\begin{proof}
The identity \eqref{lemma_commutation_diff_B} was obtained in \cite[Lemma~A.3]{ALTR16}. By applying Green's formula and
	the equality $(\alpha \cdot \nabla)^2 = \Delta$ we obtain
  \begin{equation*}
    \begin{split}
    \|  \alpha \cdot \nabla u \|_{L^2(\Omega,\CC^4)}^2 
        &= \langle \alpha \cdot \nu u, \alpha \cdot \nabla u\rangle_{L^2(\partial \Omega,\CC^4)} 
          -\langle u, (\alpha \cdot \nabla)^2 u \big\rangle_{L^2(\Omega,\CC^4)}\\
		&=\big\langle u, (\alpha \cdot \nu) \cdot (\alpha \cdot \nabla) u\big\rangle_{L^2(\partial \Omega,\CC^4)} 
          -\langle u, \Delta u \rangle_{L^2(\Omega,\CC^4)} \\
      &= \big\langle u, (\alpha \cdot \nu) \cdot (\alpha \cdot \nabla) u\big\rangle_{L^2(\partial \Omega,\CC^4)} 
      - \langle u, \nu \cdot \nabla u\rangle_{L^2(\partial \Omega,\CC^4)}  + \| \nabla u \|_{L^2(\Omega,\RR^3\otimes\CC^4)}^2,
    \end{split}
  \end{equation*}
	and one arrives  at \eqref{equation_alpha_term} with the help of \eqref{lemma_alpha_nu}. Furthermore,
	an integration by parts and the anti-commutation rule \eqref{eq_commutation} show that
  \begin{equation*} 
    \begin{split}
      \langle-\rmi \alpha \cdot \nabla u, \beta u\rangle_{L^2(\Omega,\CC^4)} &= \langle-\rmi \alpha \cdot \nu u, \beta u\rangle_{L^2(\partial\Omega,\CC^4)} 
        - \langle-\rmi u, \alpha \cdot \nabla \beta u\rangle_{L^2(\Omega,\CC^4)} \\
      &= \langle-\rmi \alpha \cdot \nu u, \beta u\rangle_{L^2(\partial\Omega,\CC^4)} - \langle\beta u, -\rmi \alpha \cdot \nabla u\rangle_{L^2(\Omega,\CC^4)},
    \end{split}
  \end{equation*}
  which implies
  \begin{equation*} 
      2 \Re \langle -\rmi \alpha \cdot \nabla u, \beta u\rangle_{L^2(\Omega,\CC^4)}
        = \langle-\rmi \alpha \cdot \nu u, \beta u\rangle_{L^2(\partial\Omega,\CC^4)}
  \end{equation*}
	and proves~\eqref{equation_beta_term1}. Finally,
  \begin{multline*}
          \big\| (-\rmi \alpha \cdot \nabla + m \beta) u \big\|_{L^2(\Omega,\CC^4)}^2 \\
        = \|  \alpha \cdot \nabla u \|_{L^2(\Omega,\CC^4)}^2 + m^2 \| \beta u \|_{L^2(\Omega,\CC^4)}^2
        + 2 \Re \langle-\rmi \alpha \cdot \nabla u, m \beta u\rangle_{L^2(\Omega,\CC^4)},
      \end{multline*}
	and using that $\beta$ is unitary, \eqref{equation_alpha_term} and~\eqref{equation_beta_term1} we arrive at \eqref{eqpart1}.
\end{proof}

We will also use the following assertion, which is a rather standard application
of the elliptic regularity argument, but we prefer to give a proof for the sake of completeness.
\begin{lem}\label{lem:dens_domtilde} For $\tau\notin\{-2,2\}$ the subspace
$\widetilde{\dom}(A_{m,\tau}) := \dom(A_{m,\tau}) \cap H^2(\mathbb{R}^3\setminus\Sigma,\CC^4)$
is dense in $\dom(A_{m,\tau})$ in the $H^1(\mathbb{R}^3\setminus\Sigma,\CC^4)$-norm.
\end{lem}

\begin{proof}
It is well-known, see, e.g., \cite[Thm. 1.5.1.2 and Thm. 2.4.2.5] {Gris}, that
there exists a bounded linear operator $E : H^{\frac{1}{2}}(\Sigma,\CC^4) \longrightarrow H^1(\Omega_+,\CC^4)$
such that for any $\xi \in H^{\frac{1}{2}}(\Sigma,\CC^4)$ one has $(E\xi)|_\Sigma = \xi$
and $E\big(H^{\frac{3}{2}}(\Sigma)\big) \subset H^2(\Omega_+)$.

Let $(u_+,u_-)\in\dom(A_{m,\tau})$. As $H^2(\Omega_\pm,\mathbb{C}^4)$
is dense in $H^1(\Omega_\pm,\mathbb{C}^4)$ with respect to the $H^1$-norm, there exist $v_\pm^\varepsilon \in H^2(\Omega_\pm,\CC^4)$ such that
$\lim_{\varepsilon\rightarrow0}\|v_\pm^\varepsilon - u_\pm\|_{H^1(\Omega_\pm,\CC^4)} = 0$.
Set $w_-^\varepsilon = v_-^\varepsilon$ and $w_+^\varepsilon = v_+^\varepsilon + E\varphi^\varepsilon$, where
$\varphi^\varepsilon$ is given by
$\varphi^\varepsilon = -(\mathcal{P}_\tau^-)^{-1}(\mathcal{P}_\tau^-v_+^\varepsilon + \mathcal{P}_\tau^+v_-^\varepsilon)$.
Note that $\varphi^\varepsilon \in H^{3/2}(\Sigma, \mathbb{C}^4)$ as $v_\pm^\varepsilon\in H^2(\Omega_\pm,\CC^4)$ and $\mathcal{P}_\tau^\pm, (\mathcal{P}_\tau^-)^{-1} \in C^2(\Sigma,\CC^{4\times 4})$. Thus,
we have $w_\pm^\varepsilon \in H^2(\Omega_\pm,\CC^4)$ due to the above properties of $E$.

We claim that $\lim_{\varepsilon \rightarrow 0} \|w_\pm^\varepsilon - u_\pm\|_{H^1(\Omega_\pm,\CC^4)} = 0$. By definition, it is clear that this is true for $w_-^\varepsilon$ so, we focus on $w_+^\varepsilon$. We have
\begin{align*}
	\|w_+^\varepsilon - u_+\|_{H^1(\Omega_+,\CC^4)} &\leq \|v_+^\varepsilon - u_+\|_{H^1(\Omega_+,\CC^4)} + \|E\varphi^\varepsilon\|_{H^1(\Omega_+,\CC^4)}\\
& \leq \|v_+^\varepsilon - u_+\|_{H^1(\Omega_+,\CC^4)} + C\|\varphi^\varepsilon\|_{H^{\frac{1}{2}}(\Sigma,\CC^4)},
\end{align*}
with a constant $C>0$ thanks to the boundedness of $E$. The first term in the right-hand side converges to zero by definition. This is also true for the second term because using the transmission condition $\mathcal{P}^-_\tau u_+ + \mathcal{P}^+_\tau u_- = 0$ we get
\begin{align*}
	\|\varphi^\varepsilon\|_{H^{\frac{1}{2}}(\Sigma,\CC^4)} &= \|(\mathcal{P}_\tau^+)^{-1}\big(\mathcal{P}_\tau^-(v_+^\varepsilon - u_+) + \mathcal{P}_\tau^+(v_-^\varepsilon - u_-)\big)\|_{H^{\frac{1}{2}}(\Sigma,\CC^4)}\\
&\leq K \big(\|v_+^\varepsilon - u_+\|_{H^1(\Omega_+,\CC^4)} + \|v_-^\varepsilon - u_-\|_{H^1(\Omega_-,\CC^4)}\big),
\end{align*}
with a constant $K>0$. Thus, the right-hand side converges to zero by hypothesis and we get $\lim_{\varepsilon \rightarrow 0} \|w_\pm^\varepsilon - u_\pm\|_{H^1(\Omega_\pm,\CC^4)} = 0$.

The only thing left to prove is that $(w_+^\varepsilon,w_-^\varepsilon)\in\widetilde{\dom}(A_{m,\tau})$ which is true if the transmission condition is verified. Indeed, we have
\begin{equation*}
\begin{split}
	\mathcal{P}_\tau^- w_+^\varepsilon + \mathcal{P}_\tau^+ w_-^\varepsilon &= \mathcal{P}_\tau^-\big(v_+^\varepsilon + \varphi^\varepsilon\big) + \mathcal{P}_\tau^+v_-^\varepsilon
=\mathcal{P}_\tau^-v_+^\varepsilon + \mathcal{P}_\tau^+v_-^\varepsilon + \mathcal{P}_\tau^-\varphi^\varepsilon\\
&=\mathcal{P}_\tau^-v_+^\varepsilon + \mathcal{P}_\tau^+v_-^\varepsilon - \mathcal{P}_\tau^-v_+^\varepsilon - \mathcal{P}_\tau^+v_-^\varepsilon=0.
\end{split}
\end{equation*}
Hence, the lemma is proved.
\end{proof}

\begin{proof}[Proof of Proposition \ref{theorem_form_B_tau_square}]
Due to Lemma~\ref{lem:dens_domtilde} it is sufficient to prove the result for the functions
$u\in \widetilde{\dom}(A_{m,\tau})$. Using \eqref{eqpart1} for $\Omega=\Omega_\pm$ we obtain
  \begin{equation*} 
    \begin{split}
      \| A_{m,\tau} u \|_{L^2(\RR^3,\CC^4)}^2 
          &= \big\|\big (-\rmi \alpha \cdot \nabla + m \beta) u_+ \big\|_{L^2(\Omega_+,\CC^4)}^2 
          + \big\| (-\rmi \alpha \cdot \nabla + m \beta) u_- \big\|_{L^2(\Omega_-,\CC^4)}^2 \\
      &= \| \nabla u \|_{L^2(\RR^3\setminus\Sigma,\RR^3\otimes\CC^4)}^2 + m^2 \| u \|_{L^2(\mathbb{R}^3,\CC^4)}^2+J_1+J_2
		\end{split}
	\end{equation*}
	with
	\begin{align*}
	J_1&= \langle -\rmi \alpha \cdot \nu u_+, m \beta u_+ \rangle_{L^2(\Sigma,\CC^4)} - \langle -\rmi \alpha \cdot \nu u_-, m \beta u_- \rangle_{L^2(\Sigma,\CC^4)},\\
	J_2&= \big\langle u_+, \rmi \gamma_5 \alpha \cdot (\nu \times \nabla) u_+ \big\rangle_{L^2(\Sigma,\CC^4)}
	- \big\langle u_-, \rmi \gamma_5 \alpha \cdot (\nu \times \nabla) u_- \big\rangle_{L^2(\Sigma,\CC^4)}.
	\end{align*}
	To simplify $J_1$ we remark first that 
  \[
       \langle -\rmi \alpha \cdot \nu u_\pm, m\beta u_\pm\rangle_{L^2(\Sigma,\CC^4)}\\
				= m \big\langle  u_\pm, \rmi (\alpha \cdot \nu)\beta u_\pm\big\rangle_{L^2(\Sigma,\CC^4)}=m \langle  u_\pm, \cB u_\pm\rangle_{L^2(\Sigma,\CC^4)},
  \]
	which yields 
	$J_1=m \big( \langle  u_+, \cB u_+\rangle_{L^2(\Sigma,\CC^4)}- \langle  u_-, \cB u_-\rangle_{L^2(\Sigma,\CC^4)}\big)$.
	  By \eqref{eqcommr1} and~\eqref{eqtrans2} we get
  \begin{align*}
        J_1&=m \big( \langle  u_+, \cB u_+\rangle_{L^2(\Sigma,\CC^4)}- \langle  u_-, \cB u_-\rangle_{L^2(\Sigma,\CC^4)}\big)\\
			      &= m\big[ \langle  u_+, \cB u_+\rangle_{L^2(\Sigma,\CC^4)}- \langle  u_-, \cB u_-\rangle_{L^2(\Sigma,\CC^4)}\\
						&\quad \qquad- \langle \cR_\tau^+ u_-, \cB u_-\rangle_{L^2(\Sigma,\CC^4)}
             + \langle u_-, \cB \cR_\tau^+ u_-\rangle_{L^2(\Sigma,\CC^4)}\big] \\
       &= m\big[ \langle  u_+, \cB u_+\rangle_{L^2(\Sigma,\CC^4)}- \langle  u_-, \cB u_-\rangle_{L^2(\Sigma,\CC^4)}\\
						&\qquad \quad- \langle u_+, \cB u_-\rangle_{L^2(\Sigma,\CC^4)}
             + \langle u_-, \cB u_+\rangle_{L^2(\Sigma,\CC^4)}
			\big] \\
      &= m \big\langle u_+ + u_-, \cB (u_+ - u_-) \big\rangle_{L^2(\Sigma,\CC^4)}\\
			&= \dfrac{2m}{\tau } \langle \cB(u_+ - u_-), \cB (u_+ - u_-) \big\rangle_{L^2(\Sigma,\CC^4)}\\
			&= \dfrac{2m}{\tau } \|u_+ - u_-\|_{L^2(\Sigma,\CC^4)}^2.
  \end{align*}
  It remains to analyze the term $J_2$.
  Making again use of the transmission condition \eqref{eqtrans2} and the 
  commutation relation~\eqref{eqcommr1} we obtain
  \begin{equation*}
    \begin{split}
      J_2&=\big\langle \cR_\tau^+ u_-, \rmi \gamma_5 \alpha \cdot (\nu \times \nabla) u_+ \big\rangle_{L^2(\Sigma,\CC^4)}
        -  \big\langle  u_-, \rmi \gamma_5 \alpha \cdot (\nu \times \nabla) \cR_\tau^- u_+ \big\rangle_{L^2(\Sigma,\CC^4)} \\
      &= \big\langle\gamma_5 u_-, \rmi  \big(\cR_\tau^- \alpha \cdot (\nu \times \nabla) - \alpha \cdot (\nu \times \nabla) \cR_\tau^-\big) u_+ \big\rangle_{L^2(\Sigma,\CC^4)} \\
      &= \frac{4\tau}{4-\tau^2} \Big\langle \gamma_5 u_-, \rmi  \big[\alpha \cdot (\nu \times \nabla), \cB\big] u_+ \Big\rangle_{L^2(\Sigma,\CC^4)}.
    \end{split}
  \end{equation*}
  With the help of \eqref{lemma_commutation_diff_B} we arrive at
	\[
	J_2=\frac{4\tau}{4-\tau^2} \langle \gamma_5 u_-, 2M \gamma_5\cB u_+ \rangle_{L^2(\Sigma,\CC^4)}
	=\frac{4\tau}{4-\tau^2} \langle u_-, 2M \cB u_+ \rangle_{L^2(\Sigma,\CC^4)}.
	\]
  Finally, using the expressions of $\cR_\tau^\pm$ and the transmission conditions we conclude
  \begin{align*}
      J_2&=\frac{4\tau}{4-\tau^2} \langle \cB u_-, M u_+ \rangle_{L^2(\Sigma,\CC^4)}+\frac{4\tau}{4-\tau^2} \langle  u_-, M  \cB u_+ \rangle_{L^2(\Sigma,\CC^4)}\\
          &=\langle \cR_\tau^+ u_-, M u_+ \rangle_{L^2(\Sigma,\CC^4)}
					- \langle u_-, M  \cR_\tau^- u_+ \rangle_{L^2(\Sigma,\CC^4)}\\
					&= \langle u_+, M  u_+ \rangle_{L^2(\Sigma,\CC^4)}- \langle  u_-, M  u_- \rangle_{L^2(\Sigma,\CC^4)},
  \end{align*}
	which completes the proof of \eqref{equation_form} for $u\in \widetilde{\dom}(A_{m,\tau})$.
\end{proof}

\subsection{First estimates for the discrete spectrum}

First remark that as a direct consequence of Corollary~\ref{corollary_no_discrete_spectrum_in_gap}
and Theorem~\ref{thm-basic}~(d) we have the following assertion:

\begin{prop}\label{prop3.4}
Let $m \in \mathbb{R}$ be fixed. Then there exists $\tau_m>0$ such that
$A_{m,\tau}$ has no discrete spectrum for $|\tau|<\tau_m$
and $|\tau|>\frac{4}{\tau_m}$.
\end{prop}

The following assertion holds true due to the unique continuation principle,
see Theorem 3.7 in~\cite{AMV15} and the discussion thereafter to obtain the result in our setting:
\begin{prop}\label{lem-amv}
Assume that $\tau\notin\{-2,2\}$ and that $\Omega_-$ is connected. Then $A_{m,\tau}$ has no eigenvalues in $\RR\setminus[-m,m]$.
\end{prop}

Now we use Proposition~\ref{theorem_form_B_tau_square} to obtain first estimates on the discrete spectrum of $A_{m,\tau}$.
\begin{prop}\label{thm-disc1} Assume that $\tau\notin\{-2,2\}$, then:
\begin{itemize}
\item[(a)] the discrete spectrum of $A_{m,\tau}$ is finite,
\item[(b)] if $m\tau\ge 0$, then the discrete spectrum of $A_{m,\tau}$ is empty,
\item[(c)] if $m\tau>0$, then $\pm m$ are not eigenvalues of $A_{m,\tau}$,
\item[(d)] if $m\tau>0$ and $\Omega_-$ is connected, then $A_{m,\tau}$ has no eigenvalues.
\end{itemize}
\end{prop}

\begin{proof}
(a) It is sufficient to show that the discrete spectrum of $A:=A_{m,\tau}^2$ is finite,
i.e. that $A$ has at most finitely many eigenvalues in $[0,m^2)$.
Recall that $A$ is the self-adjoint operator associated with the sesquilinear form
\[
a(u,u)=\|A_{m,\tau} u\|^2_{L^2(\RR^3,\CC^4)}, \quad u\in \dom(A_{m,\tau}).
\]
Let $\Omega\subset\RR^3$ be a large ball containing $\overline{\Omega_+}$ and set $\Omega^c = \RR^3\setminus\overline{\Omega}$.
Using the natural identification
\[
L^2(\RR^3,\CC^4)\simeq L^2(\Omega_+,\CC^4)\oplus L^2(\Omega_-\cap \Omega,\CC^4)\oplus L^2(\Omega^c,\CC^4),
\quad u\simeq (u_+,u_-,u_c),
\]
consider the sesquilinear form
\begin{multline*}
b(u,u)=\iiint_{\RR^3\setminus (\Sigma \cup\partial \Omega)} \big| \nabla u \big|^2\, \dd x + m^2 \iiint_{\RR^3} | u|^2\, \dd x\\
          + \frac{2 m}{\tau} \iint_\Sigma | u_+ - u_-|^2 \dd\Sigma + \iint_\Sigma M |u_+|^2 \dd\Sigma 
          - \iint_\Sigma M |u_-|^2 \dd\Sigma
\end{multline*}
defined for $u$ satisfying
\begin{gather*}
u_+\in H^1(\Omega_+,\CC^4), \quad u_-\in H^1(\Omega_-\cap \Omega,\CC^4), \quad
u_c\in H^1(\Omega^c,\CC^4),\\
\cP_\tau^- u_+ + \cP_\tau^+ u_-=0 \text{ on } \Sigma.
\end{gather*}
Then $b$ is closed, lower semibounded, and, moreover, it is an extension of the form~$a$.
Let $B$ be the self-adjoint operator in $L^2(\RR^3,\CC^4)$ associated with $b$, then due to the min-max principle
one has $\varepsilon_n(A)\ge \varepsilon_n(B)$ for all $n$. Therefore, it is sufficient to show that $B$
has at most finitely many eigenvalues in $(-\infty,m^2)$.

One easily remarks that $B=B_0\oplus B_c$, where $B_0$ is the self-adjoint operator in $L^2(\Omega,\CC^4)$
generated by the sesquilinear form
\begin{multline*}
b_0(u,u)=\iiint_{\Omega\setminus \Sigma } \big| \nabla u \big|^2\, \dd x + m^2 \iiint_{\Omega} | u|^2\, \dd x\\
          + \frac{2 m}{\tau} \iint_\Sigma | u_+ - u_-|^2 \dd\Sigma + \iint_\Sigma M |u_+|^2 \dd\Sigma 
          - \iint_\Sigma M |u_-|^2 \dd\Sigma
\end{multline*}
with
\begin{multline*}
\dom(b_0)=\big\{ u=(u_+,u_-):
u_+\in H^1(\Omega_+,\CC^4), \quad u_-\in H^1(\Omega_-\cap \Omega,\CC^4),\\
\quad \cP_\tau^- u_+ + \cP_\tau^+ u_-=0 \text{ on } \Sigma
\big\},
\end{multline*}
and $B_c$ is the self-adjoint operator in $L^2(\Omega^c,\CC^4)$ given by the sesquilinear form
\[
b_c(u_c,u_c)=\iiint_{\Omega^c } \big| \nabla u_c \big|^2\, \dd x + m^2 \iiint_{\Omega^c} | u_c|^2\, \dd x,\quad
\dom(b_c)=H^1(\Omega^c,\CC^4).
\]
One has $B_c\ge m^2$ and, therefore, the number of eigenvalues of $B$ in $(-\infty,m^2)$
is the same as that of $B_0$. On the other hand, the domain of $b_0$
is compactly embedded in $L^2(\Omega,\CC^4)$ and hence, $B_0$ has compact resolvent. As $B_0$
is lower semibounded, its eigenvalues form a sequence converging to $+\infty$
and there are at most finitely many eigenvalues in $(-\infty,m^2)$.

(b) It is sufficient to show that $A_{m,\tau}^2$ has no discrete spectrum.
As the essential spectrum of $A_{m,\tau}^2$ is $[m^2,+\infty)$,
it is sufficient to show that $A_{m,\tau}^2\ge m^2$.
The case $\tau=0$ is obvious and corresponds to the free Dirac operator, cf. Section~\ref{section_free_Operator}, 
so we may assume that $\tau\ne 0$ and $m\tau\ge 0$. By Proposition~\ref{theorem_form_B_tau_square} 
we have for any $u\in \dom (A^2_{m,\tau})\subset \dom(A_{m,\tau})$
\begin{multline}
   \label{eq-amt}
\langle u, A^2_{m,\tau} u\rangle_{L^2(\RR^3,\CC^4)}=\|A_{m,\tau}u\|^2_{L^2(\RR^3,\CC^4)}\\
=\|A_{0,\tau}u\|^2_{L^2(\RR^3,\CC^4)}
+m^2\|u\|^2_{L^2(\RR^3,\CC^4)} +\frac{2 m}{\tau} \| u_+ - u_-\|^2_{L^2(\Sigma,\CC^4)}
\ge m^2\|u\|^2_{L^2(\RR^3,\CC^4)}
\end{multline}
and thus, the claim is also true for $\tau \neq 0$.

(c) It is sufficient to show that $\ker(A^2_{m,\tau}-m^2)=\{0\}$. Let $u\in \ker(A^2_{m,\tau}-m^2)$,
then similar to \eqref{eq-amt}  one has
\begin{align*}
0&=\langle u, (A^2_{m,\tau}-m^2) u\rangle_{L^2(\RR^3,\CC^4)}\\
&=\|A_{0,\tau}u\|^2_{L^2(\RR^3,\CC^4)} +\frac{2 m}{\tau} \| u_+ - u_-\|^2_{L^2(\Sigma,\CC^4)}
\ge \frac{2 m}{\tau} \| u_+ - u_-\|^2_{L^2(\Sigma,\CC^4)}\ge 0
\end{align*}
implying $u_+=u_-$ on $\Sigma$. Together with the condition $\cP_\tau^- u_+ +\cP_\tau^+ u_-= 0$
this implies that $u_+=u_-=0$ on $\Sigma$.
Using again Proposition~\ref{theorem_form_B_tau_square} we arrive at
\[
0=\langle u, (A^2_{m,\tau}-m^2) u\rangle_{L^2(\RR^3,\CC^4)}= \iiint_{\Omega_+} |\nabla u_+|^2\dd x+\iiint_{\Omega_-} |\nabla u_-|^2\dd x
\]
and deduce that $u_\pm$ are constant on each connected component of $\Omega_\pm$. As $u_\pm=0$ on $\Sigma=\partial\Omega_\pm$,
the functions $u_\pm$ vanish identically.

(d) Follows from the points (b) and (c) and Proposition~\ref{lem-amv}.
\end{proof}

\section{Discrete spectrum in the large mass limit}\label{sec:dis-spec}

\subsection{Effective operator on the shell}\label{sec-eff}

By Theorem~\ref{thm-disc1} $A_{m, \tau}$ can only have discrete spectrum when $\tau$ 
and $m$ have opposite signs. 
As seen in Theorem~\ref{thm-basic}~(e), the operators $A_{m,-\tau}$ and $A_{-m,\tau}$ are unitarily equivalent, hence
in this section we assume without loss of generality that
\[
\tau<0 \text{ with } \tau\ne -2 \text{ is fixed}
\]
and we are going to study the behavior of the discrete spectrum as $m\to +\infty$.

In order to state the main result, we need to introduce an effective operator on $\Sigma$, which appears
to be a Schr\"odinger operator with an external Yang-Mills potential, cf. \cite[Section 69]{sred}.
Namely, consider the (matrix-valued) $1$-form $\omega$ on $\Sigma$ given by $\omega=\sigma\cdot (\nu\times\dd \nu)$,
i.e. by the local expression 
\begin{equation}
    \label{eq-omega0}
\omega=\omega_1\dd s_1+\omega_2\dd s_2 \in T^*\Sigma\otimes \bfB(\CC^2), \quad
\omega_j=\sigma\cdot(\nu\times \partial_j \nu).
\end{equation}
For a parameter (coupling constant) $\theta\in \RR$, denote
\[
\Lambda(\theta)=(\dd + \rmi \theta \omega)^*(\dd + \rmi \theta \omega)
\]
the associated Bochner Laplacian in $L^2(\Sigma,\CC^2)$.
Recall that by definition this operator is given by the local expression
\[
\Lambda(\theta)=
-\dfrac{1}{\sqrt{\det g}}\sum_{j,k} (\partial_j + \rmi \theta\omega_j)g^{jk} 
\sqrt{\det g} (\partial_k + \rmi \theta\omega_k), \quad \dom\big(\Lambda(\theta)\big)=H^2(\Sigma,\CC^2),
\]
where $(g_{jk})$ is the metric tensor on $\Sigma$, $(g^{jk}):=(g_{jk})^{-1}$,
and it is the unique self-adjoint operator associated with the closed sesqulinear form $\lambda_\theta$ given by
\[
\lambda_\theta(u,u)= \iint_\Sigma \sum_{j,k}g^{jk} \big\langle \partial_j u+\rmi \theta \omega_j u, 
  \partial_k u+\rmi \theta \omega_k u\big\rangle_{\CC^2} \dd\Sigma,
  \quad  \dom(\lambda_\theta)=H^1(\Sigma,\CC^2).
\]
Finally, consider the Schr\"odinger operator with an additional (bounded) scalar potential induced by curvatures given by
\[
\Upsilon_\tau=\Lambda\Big(\frac{4}{\tau^2+4}\Big) - \Big(\dfrac{\tau^2-4}{\tau^2+4}\Big)^2 M^2 + \dfrac{\tau^4+16}{(\tau^2+4)^2} K,
\]
which acts on $L^2(\Sigma,\CC^2)$ as well.

We will often use the shorthand
\begin{equation}\label{eq-aaa}
\mu=\dfrac{4|\tau|}{\tau^2+4}\in (0,1).
\end{equation}

The aim of the present section is prove the following main result:
\begin{thm}\label{thm-ev1}
Assume that $\delta\equiv \delta(m)>0$ is chosen in such a way that
\begin{equation}
    \label{eq-mdelta}
 \delta\to 0 \text { and }
m\delta\to +\infty
\text{ for } m\to +\infty.
\end{equation}
Then there exist constants $b>0$, $c>0$ and $m_0>0$ such that for all $m>m_0$ and $j\in\big\{1,\dots,\cN(A^2_{m,\tau},m^2)\big\}$
one has
\begin{multline*}
\Big( \dfrac{\tau^2-4}{\tau^2+4}\Big)^2 m^2 + (1-b\delta)E_j(\Upsilon_\tau\oplus\Upsilon_\tau)- c(\delta+m^2e^{-2\mu m\delta})
\le E_j(A_{m,\tau}^2)\\
\le\Big( \dfrac{\tau^2-4}{\tau^2+4}\Big)^2 m^2 + (1+b\delta)E_j(\Upsilon_\tau\oplus\Upsilon_\tau)+ c(\delta+m^2e^{-2\mu m\delta}).
\end{multline*}
\end{thm}

Let us present first some important consequences.

\begin{cor}\label{cor11} For any fixed $j\in\NN$ there holds
\[
E_j(A_{m,\tau}^2)= \Big(\dfrac{\tau^2-4}{\tau^2+4}\Big)^2 m^2 + E_j(\Upsilon_\tau\oplus\Upsilon_\tau)+\cO\Big(\dfrac{\log m}{m}\Big)
\text{ as } m\to+\infty.
\]
\end{cor}

\begin{proof}
As the $j$-th eigenvalue of $\Upsilon_\tau\oplus\Upsilon_\tau$ does not depend on $m$, it is sufficient to use Theorem~\ref{thm-ev1} with $\delta= k m^{-1} \log m$
and a sufficiently large $k>0$.
\end{proof}

\begin{cor}\label{cor32} Denote the eigenvalues of $A_{m,\tau}$ by $\pm \mu_j(m)$ with $\mu_j(m)\ge 0$ enumerated in the non-decreasing order
according to the multiplicities, then
for any fixed $j\in\NN$ there holds
\[
\mu_j(m)=\dfrac{|\tau^2-4|}{\tau^2+4}\, m+ \dfrac{\tau^2+4}{|\tau^2-4|}\, \dfrac{E_{j}(\Upsilon_\tau)}{2 m} + \cO\Big(\dfrac{\log m}{m^2}\Big)
\text{ as } m\to+\infty.
\]
\end{cor}

\begin{proof}
One has $\mu_j(m)^2=E_{2j}(A_{m,\tau}^2)$ due to the degeneracy, see Theorem \ref{thm-basic}(c).
Now it is sufficient to apply Taylor expansion to $\sqrt{E_{2j}(A_{m,\tau}^2)}$ using the asymptotics of Corollary~\ref{cor11}
and to remark that $E_{2j}(\Upsilon_\tau\oplus\Upsilon_\tau)=E_j(\Upsilon_\tau)$.
\end{proof}

Finally, the following Weyl-type asymptotics holds:
\begin{cor}\label{cor33}
The total number $\cN(A^2_{m,\tau},m^2)$ of discrete eigenvalues of $A_{m,\tau}$ satisfies
\[
\cN(A^2_{m,\tau},m^2)= \dfrac{16}{\pi} \dfrac{\tau^2}{(\tau^2+4)^2} \,|\Sigma| \,m^2+\cO(m\log m) \text{ as } m\to +\infty.
\]
\end{cor}

\begin{proof}
Using Theorem~\ref{thm-ev1} with $\delta= k m^{-1} \log m$ and a sufficiently large $k>0$ one concludes that
there exist constants $C>0$ and $m_0>0$ such that for $m>m_0$ there holds
\begin{multline}
   \label{eq-weyl1}
\cN\Big(\Upsilon_\tau\oplus\Upsilon_\tau,\dfrac{16 \tau^2}{(\tau^2+4)^2} m^2-Cm \log m\Big)\\
\le\cN(A^2_{m,\tau},m^2) \le
\cN\Big(\Upsilon_\tau\oplus\Upsilon_\tau, \dfrac{16 \tau^2}{(\tau^2+4)^2} m^2+Cm \log m\Big).
\end{multline}
Due to the obvious identity $\cN(\Upsilon_\tau\oplus\Upsilon_\tau,E)\equiv 2\cN(\Upsilon_\tau,E)$
it is sufficient to study the behavior of $\cN(\Upsilon_\tau,E)$ for large $E$. As $\Upsilon_\tau$
is an elliptic differential operator on a compact manifold having the same principal symbol as the Laplacian,
the classical Weyl asymptotics, see e.g. Section 16.1 in \cite{S01}, gives
\[
\cN(\Upsilon_\tau,E)=2\cdot \dfrac{|\Sigma|}{4\pi }\,E + \cO(\sqrt{E}), \quad E\to+\infty,
\]
and the substitution into \eqref{eq-weyl1} gives the result.
We remark that the latter result on $\cN(\Upsilon_\tau,E)$ is indeed very standard for the operators with $C^\infty$ coefficients, but in our case
the coefficients are only supposed to be $C^2$. For the extension of the Weyl asymptotics to $C^k$ coefficients see e.g. Theorem~1.1 in \cite{Z99}.
\end{proof}

\begin{remark}
One easily sees that $\Upsilon_\tau$ commutes with the charge conjugation operator $u \mapsto \sigma_2 \overline{u}$
satisfying $\langle u,  \sigma_2 \overline{u}\rangle_{L^2(\Sigma,\CC^2)}=0$ for any $u\in {L^2(\Sigma,\CC^2)}$.
This implies that any eigenvalue of $\Upsilon_\tau$ has an even multiplicity, which is in agreement with Theorem~\ref{thm-basic}~(c).
Furthermore, a short direct computation shows that the operators $\Lambda(\theta)$ and $\Lambda(1-\theta)$ are unitarily equivalent, the associated unitary operator
being $u\mapsto (\sigma\cdot \nu) u$. As a result, the operator $\Upsilon_\tau$ is unitarily equivalent to $\Upsilon_{\frac{4}{\tau}}$, which is in agreement with 
Theorem~\ref{thm-basic}~(d). 
\end{remark}

\subsection{Intermediate operator}\label{ssec42}

In what follows, it will be more comfortable to work with another operator which is unitarily equivalent to $\Upsilon_\tau\oplus\Upsilon_\tau$
but acts in a different space. Namely, consider the following Hilbert space:
\begin{equation} \label{def_Hilbert_space_H}
\begin{split}
\cH\equiv\cH_\tau=&\big\{v=(v_+,v_-): v_\pm \in L^2(\Sigma,\CC^4): \quad
\cP^-_\tau v_+ + \cP^+_\tau v_-=0\big\},\\
&\langle u,v\rangle_{\cH}=\langle u_+,v_+\rangle_{L^2(\Sigma,\CC^4)}+\langle u_-,v_-\rangle_{L^2(\Sigma,\CC^4)},
\end{split}
\end{equation}
and denote by $\cL^\tau_0$ the self-adjoint operator associated with the sesquilinear form
\begin{equation}
   \label{eq-ll1}
\begin{aligned}
\ell^\tau_0(v,v)&=\iint_\Sigma \Big(\|\nabla_s v_+\|^2_{T_s\Sigma\otimes\CC^4}+\|\nabla_s v_-\|^2_{T_s\Sigma\otimes\CC^4} \Big)\dd\Sigma,\\
 \dom(\ell^\tau_0)&=\big\{v=(v_+,v_-)\in \cH_\tau: v_\pm \in H^1(\Sigma,\CC^4)\big\},
\end{aligned}
\end{equation}
where $\nabla_s v$ stands for the gradient of $v$ on $\Sigma$.

\begin{prop}\label{prop45} 
 The operators $\Upsilon_\tau\oplus\Upsilon_\tau$  and $\cL^\tau:=\cL^\tau_0+K-M^2$
are unitarily equivalent.
\end{prop}

\begin{proof}
As the matrices $\cP^\pm_\tau(s)$ are invertible for any $s \in \Sigma$, the map $V:L^2(\Sigma,\CC^4)\to \cH$ given by
\begin{equation*}
\begin{split}
(V f)_\pm(s)&=\mp \cP^\pm_\tau(s) f(s) = \mp \Big(\dfrac{\tau}{2}\pm \cB(s)\Big)f(s)\\
&= \Big(-\cB(s)\mp \dfrac{\tau}{2}\Big) f(s) = \Big(i\beta \alpha \nu (s)\mp \dfrac{\tau}{2}\Big) f(s)
\end{split}
\end{equation*}
is bijective. Furthermore,  everywhere on $\Sigma$ one has
\begin{align*}
|(V f)_\pm|^2=\Big|\big(\cB\pm \dfrac{\tau}{2}\big) f \Big|^2& =|\cB f|^2+ \dfrac{\tau^2}{4} |f|^2 \pm \tau \Re \langle \cB f,f\rangle\\
&=\Big(1+\dfrac{\tau^2}{4}\Big) |f|^2  \pm \tau \Re \langle \cB f,f\rangle,
\end{align*}
and then
\begin{align*}
\|(V f)_\pm\|^2_{L^2(\Sigma,\CC^4)}&=\dfrac{\tau^2+4}{4} \|f\|^2_{L^2(\Sigma,\CC^4)}\pm \tau \Re \langle \cB f,f\rangle_{L^2(\Sigma,\CC^4)},\\
\|V f\|^2_\cH&=\|(V f)_+\|^2_{L^2(\Sigma,\CC^4)}+\|(V f)_-\|^2_{L^2(\Sigma,\CC^4)}= \dfrac{\tau^2+4}{2} \|f\|^2_{L^2(\Sigma,\CC^4)}.
\end{align*}
Therefore, the operator
\[
U:=\sqrt{\dfrac{2}{\tau^2+4}}\, V:L^2(\Sigma,\CC^4)\to \cH
\]
is unitary. We are going to show that
$U^*\cL^\tau U=\Upsilon_\tau \oplus \Upsilon_\tau$. As the operators $K$, $M$, $U$ act pointwise, they commute and thus
\begin{equation}
  \label{eq-ula}
U^*\cL^\tau U=U^*\cL^\tau_0 U+K-M^2.
\end{equation}

In order to obtain an expression for $U^*\cL^\tau_0 U$ let us transform the expression $\ell^\tau_0(U f, Uf)$ for $f\in H^1(\Sigma,\CC^4)$.
In the local coordinates of $\Sigma$ one has
\begin{multline*}
\big\|\nabla_s (Uf)_\pm\big\|^2_{T_s\Sigma\otimes\CC^4}
=\sum_{j,k}g^{jk} \big\langle \partial_j (Uf)_\pm, \partial_k (Uf)_\pm\big\rangle_{\CC^4}\\
=\dfrac{2}{\tau^2+4}\sum_{j,k}g^{jk} 
\big\langle \rmi \beta \alpha \cdot  \nu \partial_j f + \rmi \beta \alpha \cdot \partial_j\nu  f \mp \dfrac{\tau}{2}\partial_j f,
\rmi \beta \alpha \cdot  \nu \partial_k f + \rmi \beta \alpha \cdot \partial_k\nu  f \mp \dfrac{\tau}{2}\partial_k f\big\rangle_{\CC^4}
\end{multline*}
and
\begin{multline*}
\big\langle \rmi \beta \alpha \cdot \nu \partial_j f + \rmi \beta \alpha \cdot \partial_j\nu  f \mp \dfrac{\tau}{2}\partial_j f,
\rmi \beta \alpha \cdot \nu \partial_k f + \rmi \beta \alpha \cdot \partial_k\nu  f \mp \dfrac{\tau}{2}\partial_k f\big\rangle_{\CC^4}\\
=\big\langle \rmi \beta \alpha \cdot \nu \partial_j f + \rmi \beta \alpha \cdot \partial_j\nu  f,
\rmi \beta \alpha \cdot \nu \partial_k f + \rmi \beta \alpha \cdot \partial_k\nu  f\big\rangle_{\CC^4}
+\dfrac{\tau^2}{4} \langle \partial_j f, \partial_k f\rangle_{\CC^4}\\
\mp \dfrac{\tau}{2}\Big(\big\langle \partial_j f, \rmi \beta \alpha \cdot \nu \partial_k f + \rmi \beta \alpha \cdot \partial_k\nu  f\big\rangle_{\CC^4}
+\big\langle \rmi \beta \alpha \cdot \nu \partial_j f + \rmi \beta \alpha \cdot \partial_j\nu  f, \partial_k f\big\rangle_{\CC^4}\Big).
\end{multline*}
It follows that
\begin{equation*}
\begin{split}
&\big\|\nabla_s (Uf)_+\big\|^2_{T_s\Sigma\otimes\CC^4}+\big\|\nabla_s (Uf)_-\big\|^2_{T_s\Sigma\otimes\CC^4}\\
&=\dfrac{4}{\tau^2+4}\sum_{j,k}g^{jk} \Big(\big\langle \rmi \beta \alpha \cdot \nu \partial_j f + \rmi \beta \alpha \cdot \partial_j\nu  f,
\rmi \beta \alpha \cdot \nu \partial_k f + \rmi \beta \alpha \cdot \partial_k\nu  f\big\rangle_{\CC^4}
+\dfrac{\tau^2}{4} \langle \partial_j f, \partial_k f\rangle_{\CC^4}\Big).
\end{split}
\end{equation*}
We further use the unitarity of $\beta$ and of $\alpha\cdot\nu$ to transform
\begin{equation*}
\begin{split}
\big\langle \rmi \beta \alpha \cdot \nu \partial_j f + &\rmi \beta \alpha \cdot \partial_j\nu  f,
\rmi \beta \alpha \cdot \nu \partial_k f + \rmi \beta \alpha\cdot  \partial_k\nu  f\big\rangle_{\CC^4} \\
&=\big\langle \alpha \cdot \nu \partial_j f + \alpha \cdot \partial_j\nu  f,
\alpha \cdot \nu \partial_k f + \alpha \cdot \partial_k\nu  f\big\rangle_{\CC^4}\\
&=\big\langle \partial_j f + (\alpha \cdot  \nu ) (\alpha \cdot  \partial_j\nu)  f,
\partial_k f + (\alpha \cdot  \nu ) (\alpha \cdot  \partial_k\nu)  f\big\rangle_{\CC^4}=:J.
\end{split}
\end{equation*}
Now we use the identity 
\begin{equation}
   \label{eq-axy}
(\alpha\cdot x) (\alpha\cdot y)=(x\cdot y) I_4 + \rmi \gamma_5 \alpha \cdot (x\times y)
\end{equation}
and the equality $\nu \cdot \partial_j \nu=0$, which holds due to $|\nu|=1$, to find
\[
J=\big\langle \partial_j f + \rmi \gamma_5 \alpha \cdot (\nu\times \partial_j \nu)  f,
\partial_k f + \rmi \gamma_5 \alpha \cdot (\nu\times \partial_k \nu)  f\big\rangle_{\CC^4}.
\]
Denote
\begin{equation}
     \label{eq-aaj}
A_j:=\gamma_5 \alpha \cdot (\nu\times \partial_j \nu),
\end{equation}
then we have 
\begin{multline*}
\big\|\nabla_s (Uf)_+\big\|^2_{T_s\Sigma\otimes\CC^4}+\big\|\nabla_s (Uf)_-\big\|^2_{T_s\Sigma\otimes\CC^4}\\
=\dfrac{4}{\tau^2+4}\sum_{j,k}g^{jk} \Big(
\big\langle \partial_j f + \rmi A_j  f, \partial_k f + \rmi A_k f\big\rangle_{\CC^4}
+\dfrac{\tau^2}{4} \langle \partial_j f, \partial_k f\rangle_{\CC^4}\Big).
\end{multline*}
Because of
\begin{equation*}
\begin{split}
&\big\langle \partial_j f + \rmi A_j  f, \partial_k f + \rmi A_k f\big\rangle_{\CC^4}
+\dfrac{\tau^2}{4} \langle \partial_j f, \partial_k f\rangle_{\CC^4}\\
&\qquad \qquad \qquad =\dfrac{\tau^2+4}{4} \langle \partial_j f, \partial_k f\rangle_{\CC^4} + 
\big\langle \partial_j f , \rmi A_k f\big\rangle_{\CC^4}+\big\langle \rmi A_j  f, \partial_k f \big\rangle_{\CC^4}
+\big\langle \rmi A_j  f, \rmi A_k f\big\rangle_{\CC^4}\\
&\qquad \qquad \qquad =\dfrac{\tau^2+4}{4}\Big( \langle \partial_j f, \partial_k f\rangle_{\CC^4}
+\big\langle \partial_j f , \rmi \dfrac{4}{\tau^2+4}A_k f\big\rangle_{\CC^4}+\big\langle \rmi \dfrac{4}{\tau^2+4} A_j  f, \partial_k f \big\rangle_{\CC^4}\\
&\qquad \qquad \qquad \qquad \qquad \qquad +\dfrac{\tau^2+4}{4}\big\langle \rmi \dfrac{4}{\tau^2+4} A_j  f, \rmi \dfrac{4}{\tau^2+4}A_k f\big\rangle_{\CC^4}\Big)\\
&=\dfrac{\tau^2+4}{4} \bigg( \Big\langle \partial_j f+\rmi \dfrac{4}{\tau^2+4} A_j f, \partial_k f+\rmi \dfrac{4}{\tau^2+4} A_k f\Big\rangle_{\CC^4}
+\dfrac{\tau^2}{4}\cdot \Big( \dfrac{4}{\tau^2+4}\Big)^2\big\langle \rmi A_j  f, \rmi A_k f\big\rangle_{\CC^4}\bigg),
\end{split}
\end{equation*}
we obtain
\begin{equation}
   \label{eq-ll3}
\begin{split}
\big\|&\nabla_s (Uf)_+\big\|^2_{T_s\Sigma\otimes\CC^4}+\big\|\nabla_s (Uf)_-\big\|^2_{T_s\Sigma\otimes\CC^4}\\
&=\sum_{j,k}g^{jk} \bigg( \big\langle \partial_j f+\rmi \dfrac{4}{\tau^2+4} A_j f, \partial_k f+\rmi \dfrac{4}{\tau^2+4} A_k f\big\rangle_{\CC^4}
+\dfrac{4\tau^2}{(\tau^2+4)^2}\big\langle A_j  f, A_k f\big\rangle_{\CC^4}\bigg)\\
&=\sum_{j,k}g^{jk}  \bigg( \big\langle \partial_j f+\rmi \dfrac{4}{\tau^2+4} A_j f, \partial_k f+\rmi \dfrac{4}{\tau^2+4} A_k f\big\rangle_{\CC^4}\bigg)
+ \dfrac{4\tau^2}{(\tau^2+4)^2} \big\langle f, W f\big\rangle_{\CC^4}
\end{split}
\end{equation}
with 
\begin{equation*}
  W:=\sum_{j,k}g^{jk} A_j A_k.
\end{equation*}
where we used $A_j^*=A_j$, which holds by \eqref{eq-aaj} and \eqref{commutation_gamma_5}.
Using the expression for $A_j$ and \eqref{eq-axy} we obtain
\begin{equation}
    \label{eq-ajak}
\begin{split}
A_j A_k&=\gamma_5 \big(\alpha \cdot(\nu\times\partial_j \nu)\big)\gamma_5 \big(\alpha \cdot(\nu\times\partial_k \nu)\big)\\
&=\big(\alpha \cdot(\nu\times\partial_j \nu)\big)\big(\alpha \cdot(\nu\times\partial_k \nu)\big)
= a_{jk} I_4+ \rmi \gamma_5 \alpha\cdot b_{jk}
\end{split}
\end{equation}
with
\[
a_{jk}:=(\nu\times\partial_j \nu)\cdot (\nu\times\partial_k \nu),
\quad
b_{jk}= (\nu\times\partial_j \nu)\times(\nu\times\partial_k \nu).
\]
Due to $g^{jk}=g^{kj}$ and $b_{kj}=-b_{jk}$ we have $\sum_{jk}g^{jk}b_{jk}=0$, which shows that
$W$ is a scalar potential,
\[
W=\sum_{j,k}g^{jk} a_{jk} I_4.
\]		
Recall the elementary identities
\begin{equation}
   \label{cross1}
\begin{aligned}
(a\times b)\cdot (c\times d)&= (a\cdot c) (b\cdot d) -(a\cdot d)(b\cdot c),\\
(a\times b)\times (a\times c)&= \big(a\cdot(b\times c)\big) a, \quad a,b,c,d\in\RR^3,
\end{aligned}
\end{equation}
then 
$a_{jk}= |\nu|^2 (\partial_j\nu\cdot\partial_k\nu) - (\nu \cdot\partial_j\nu) (\nu \cdot\partial_k\nu) = \partial_j\nu\cdot\partial_k\nu$, as $|\nu|=1$.
In order to give a more explicit expression for $W$
we assume that the local coordinates are chosen in such a way that the associated tangent
vectors $t_j$ correspond to the principal directions,
i.e. that $\partial_j \nu=\kappa_j t_j$ with $\kappa_j$
being the principal curvatures, then $g_{jk}$ and $g^{jk}$ are diagonal,
$\partial_j\nu\cdot\partial_k\nu=\kappa_j \kappa_k g_{jk} \delta_{jk}$, and 
\[
W=\sum_{j,k} \kappa_j \kappa_k\delta_{jk} g^{jk} g_{jk}=\kappa_1^2+\kappa_2^2 = 4M^2-2K.
\]
Therefore, it follows from \eqref{eq-ll3} that
\begin{multline}
  \label{eq-elluu}
\ell^\tau_0(Uf,Uf)=\iint_\Sigma \sum_{j,k}g^{jk} \big\langle \partial_j f+\rmi \dfrac{4}{\tau^2+4} A_j f, \partial_k f+\rmi \dfrac{4}{\tau^2+4} A_k f\big\rangle_{\CC^4}\dd\Sigma\\
+\dfrac{8\tau^2}{(\tau^2+4)^2} \big\langle f,(2 M^2 - K)f\big\rangle_{L^2(\Sigma,\CC^4)}.
\end{multline}
Furthermore, using in \eqref{eq-aaj} the expressions \eqref{def_Dirac_matrices} for the Dirac matrices 
one has
\begin{equation}
   \label{eq-aomega}
\gamma_5 \alpha_j=\begin{pmatrix}
\sigma_j & 0\\
0 & \sigma_j
\end{pmatrix}, \quad
A_j= \begin{pmatrix}
\omega_j & 0\\
0 & \omega_j
\end{pmatrix}
\end{equation}
with $\omega_j$ given in \eqref{eq-omega0}. Therefore, using the natural unitary identification operator
$J:L^2(\Sigma,\CC^2)\otimes L^2(\Sigma,\CC^2)\to L^2(\Sigma,\CC^4)$ one may rewrite \eqref{eq-elluu}
as
\begin{multline*}
\ell^\tau_0(Uf,Uf)=(\lambda_\frac{4}{\tau^2+4} \oplus \lambda_\frac{4}{\tau^2+4})(J^*f,J^*f)\\
+\dfrac{8\tau^2}{(\tau^2+4)^2} \big\langle J^*f,(2 M^2 - K)J^*f\big\rangle_{L^2(\Sigma,\CC^2)\otimes L^2(\Sigma,\CC^2)},
\end{multline*}
which yields
\[
(U J)^*\cL_0^\tau (UJ) = \Lambda\Big(\frac{4}{\tau^2+4}\Big)\oplus \Lambda\Big(\frac{4}{\tau^2+4}\Big)+\dfrac{8\tau^2}{(\tau^2+4)^2} (2M^2-K).
\]
As $K$, $M$ and $J$ commute, the substitution into \eqref{eq-ula} completes the proof.
\end{proof}

As both $K$ and $M$ are bounded, one can set $c_0:=\|K-M^2\|_\infty$
and remark that for all $c>0$, $\delta>0$ and $j\in \NN$ there holds
\begin{align*}
E_j\big((1+c\delta) \cL^\tau_0+K-M^2\big)&\equiv  E_j\Big((1+c\delta) (\cL^\tau_0+K-M^2) -c\delta(K-M^2)\Big)\\
&\le (1+c\delta) E_j(\Upsilon_\tau \oplus \Upsilon_\tau) c_0 c\delta,\\
E_j\big((1-c\delta) \cL^\tau_0+K-M^2\big)&\equiv  E_j\Big((1-c\delta) (\cL^\tau_0+K-M^2) +c\delta(K-M^2)\Big)\\
&\ge (1-c\delta) E_j(\Upsilon_\tau \oplus \Upsilon_\tau)- c_0 c\delta.
\end{align*}
Therefore, Theorem~\ref{thm-ev1} becomes a consequence of the following two-side asymptotic estimate:
\begin{prop}\label{prop-main1}
Assume that $\delta\equiv \delta(m)>0$ is chosen in order to satisfy \eqref{eq-mdelta},
then there exist constants $c>0$ and $m_0>0$ such that for any $m>m_0$ and $j\in\big\{1,\dots,\cN(A^2_{m,\tau},m^2)\big\}$
it holds
\begin{multline} \label{main_estimates}
\Big( \dfrac{\tau^2-4}{\tau^2+4}\Big)^2 m^2 + E_j\Big( (1-c\delta)\cL^\tau_0 + K-M^2\Big)-c(\delta+ m^2e^{-2\mu m\delta}) \le E_j(A_{m,\tau}^2)\\
\le \Big( \dfrac{\tau^2-4}{\tau^2+4}\Big)^2 m^2 + E_j\Big( (1+c\delta)\cL^\tau_0 + K-M^2\Big)+c(\delta+m^2e^{-2\mu m\delta}).
\end{multline}
\end{prop}

The proof of Proposition~\ref{prop-main1} occupies the rest of the paper and is split into several parts.
In Subsection \ref{sec-tub} we give first
a two-side estimate for the eigenvalues $E_j(A_{m,\tau}^2)$ in terms of operators
in $\Sigma\times(-\delta,\delta)$ by using tubular coordinates. In Subsection \ref{sec-upp1}
we obtain the right-hand side inequality of \eqref{main_estimates}, and Subsection~\ref{sec-low1} is devoted to the lower bound.

One may use a part of the computation of Proposition~\ref{prop45} to give some additional information
on the external Yang-Mills potential given by the form  $\omega$ and appearing in the definition of the effective operator $\Upsilon_\tau$.

\begin{prop}\label{prop-field1}
Let $\omega$ be given by \eqref{eq-omega0} and let $\theta \in \RR$. Then
the field strength $F_\theta$ defined by $\theta\,\omega$ is $F^\theta= 2\theta (1-\theta) K (\sigma \cdot \nu) \vol_\Sigma$
with $\vol_\Sigma$ being the volume form on $\Sigma$.
\end{prop}

\begin{proof}
By definition, see \cite[Section 69]{sred}, the field strength $F\equiv F_\theta$ defined by the form $\theta \omega$ is given by
$F=F_{12}\,\dd s_1\wedge \dd s_2$, where $F_{jk}=\theta(\partial_j \omega_k-\partial_k \omega_j) + \rmi \theta^2 (\omega_j \omega_k - \omega_k \omega_j)$. One easily computes $\partial_j \omega_k= \sigma \cdot (\partial_j \nu \times \partial_k \nu)+ \sigma \cdot (\nu \times \partial_j \partial_k \nu)$, which gives $\partial_1 \omega_2-\partial_2 \omega_1= 2\sigma\cdot  x$
with $x:=\partial_1 \nu \times \partial_2 \nu$, and
$\rmi(\omega_j \omega_k - \omega_k \omega_j)=-2 \sigma  \cdot b_{jk}$ in view of \eqref{eq-ajak} and of the block representation \eqref{eq-aomega}.
To obtain a readable expression for $b_{jk}$ we use \eqref{cross1}, then $b_{jk}=\big(\nu \cdot (\partial_j \nu \times \partial_k \nu)\big)\nu$, and $b_{12}=(\nu\cdot x) \nu$ is the orthogonal projection of $x$ onto the line directed by $\nu$.
As the vectors $\partial_j  \nu$ are orthogonal to $\nu$, the vector $x$ is a multiple of $\nu$, therefore, we have
$b_{12}=x$ and $F_{12}= 2\theta(1-\theta) \sigma \cdot x$.
In order to compute the vector $x$ we assume that
the local coordinates are chosen in such a way that
the triple $(t_1,t_2,\nu)$ is direct and
recall that $\partial_j \nu=S t_j$ with $S$
being the Weingarten map, then $x=\partial_1 \nu \times \partial_2 \nu= (\det S) (t_1\times t_2)\equiv K |t_1\times t_2|\nu$.
Having in mind that the volume form is $\vol_\Sigma=|t_1\times t_2|\dd s_1\wedge \dd s_2$,
we arrive at the sought representation.
\end{proof}

\subsection{Reduction to tubular neighborhoods}\label{sec-tub}

The proof of Proposition~\ref{prop-main1} is based on a variant of rather standard estimates
in thin neighborhoods of $\Sigma$. We are going to start with the following result:
\begin{lem}\label{lem-qnd}
There exist $\delta_0>0$ and $c>0$ such that for any $\delta\in(0,\delta_0)$, any $m\in\RR$
and any $j\in\big\{1,\dots,\cN(A^2_{m,\tau},m^2)\big\}$  there holds
\[
E_j(q^N_{m,\tau,\delta})+m^2\le E_j(A^2_{m,\tau})\le E_j(q^D_{m,\tau,\delta})+m^2,
\]
where the sesquilinear forms $q^{N/D}_{m,\delta}$ in $L^2\big(\Sigma\times(-\delta,\delta),\CC^4\big)$ are given by
\begin{multline*}
q^N_{m,\tau,\delta}(u,u)=\iiint_{\Sigma\times(-\delta,\delta)}  \Big((1-c\delta)\|\nabla_s u\|^2_{T_s\Sigma\otimes\CC^4}
+ (K-M^2-c\delta)\, |u|^2\Big)\dd\Sigma\dd t\\
 + \iint_\Sigma \bigg(\int_{-\delta}^\delta |\partial_t u|^2 \dd t  + \frac{2m}{\tau} \big|u(\cdot,0^+)-u(\cdot,0^-)\big|^2 -c \big|u(\cdot,\delta)\big|^2-c\big|u(\cdot,-\delta)\big|^2\bigg) \dd\Sigma
\end{multline*}
with domain
\begin{multline*}
\dom(q^N_{m,\tau,\delta})=\Big\{ u\in H^1\Big((\Sigma\times(-\delta,\delta)) \setminus (\Sigma\times\{0\}), \mathbb{C}^4\Big): 
\cP_\tau^- u(\cdot,0^+)+\cP_\tau^+ u(\cdot,0^-)=0\Big\}
\end{multline*}
and
\begin{multline*}
q^D_{m,\tau,\delta}(u,u)=\iiint_{\Sigma\times (-\delta,\delta)} (1 + c \delta) \|\nabla_s u\|^2_{T_s\Sigma\otimes\CC^4}
+ (K-M^2+c\delta)\,|u|^2\Big)\dd\Sigma\dd t\\
+ \iint_\Sigma \bigg(\int_{-\delta}^\delta |\partial_t u|^2 \dd t  + \frac{2m}{\tau} \big|u(\cdot,0^+)-u(\cdot,0^-)\big|^2 \bigg) \dd\Sigma
\end{multline*}
with domain
\begin{multline*}
\dom(q^D_{m,\tau,\delta})=\Big\{ u\in H^1\Big(\big(\Sigma\times(-\delta,\delta)\big) \setminus (\Sigma\times\{0\}), \mathbb{C}^4\Big):\\
\cP_\tau^- u(\cdot,0^+)+\cP_\tau^+ u(\cdot,0^-)=0, \ u(\cdot,\delta)=u(\cdot,-\delta)=0\Big\}.
\end{multline*}
\end{lem}

\begin{proof}
The computations are quite standard, but we prefer to give full details
for the sake of completeness. Consider the map
	\begin{equation}\label{eqn:chang_var}
		\Phi: \Sigma\times\RR\ni (s,t)\mapsto s - t\nu (s) \in \RR^3.
	\end{equation}
According to a classical result of differential geometry there is some $\delta_0>0$
such that for all $0<\delta<\delta_0$ the mapping $\Phi:\Sigma\times (-\delta,\delta)\mapsto \Omega^\delta:=\big\{x\in \RR^3: \dist(x,\Sigma)<\delta)\big\}$
is a diffeomorphism with $\dist\big(\Phi(s,t), \Sigma\big)=|t|$. From now we assume that $\delta\in(0,\delta_0)$ and define
\[
\Phi\big(\Sigma\times(0,\delta)\big):=\Omega^\delta_+,
\quad
\Phi\big(\Sigma\times(-\delta,0)\big):=\Omega^\delta_-,
\quad
\Omega^\delta_\pm:=\Omega^\delta\cap \Omega_\pm.
\]

Denote by $a$ the sesquilinear form defined on $\dom(a)=\dom(A_{m,\tau})$ by
\begin{multline*}
a(u,u)=\|A_{m,\tau} u\|^2_{L^2(\RR^3,\CC^4)}-m^2\|u\|^2_{L^2(\RR^3,\CC^4)}\\
=\iiint_{\RR^3\setminus \Sigma} \big|\nabla u\big|^2\, \dd x + \frac{2 m}{\tau} \iint_\Sigma | u_+ - u_-|^2 \dd\Sigma + \iint_\Sigma M |u_+|^2 \dd\Sigma - \iint_\Sigma M |u_-|^2 \dd\Sigma.
\end{multline*}
Furthermore, using the natural identification
\[
L^2(\RR^3,\CC^4)\simeq L^2(\Omega^\delta_+,\CC^4)\oplus L^2(\Omega^\delta_-,\CC^4)\oplus L^2(\RR^3\setminus \overline{\Omega^\delta},\CC^4),
\quad u\simeq (u_+,u_-,u_c),
\]
consider the sesquilinear form
\begin{multline*}
b^N(u,u)=\iiint_{\Omega^\delta_+\cup \Omega^\delta_-\cup(\RR^3\setminus \partial\Omega^\delta)} \big| \nabla u \big|^2\, \dd x \\
          + \frac{2 m}{\tau} \iint_\Sigma | u_+ - u_-|^2 \dd\Sigma + \iint_\Sigma M |u_+|^2 \dd\Sigma 
          - \iint_\Sigma M |u_-|^2 \dd\Sigma
\end{multline*}
defined on the functions $u$ with
\begin{gather*}
u_+\in H^1(\Omega^\delta_+,\CC^4), \quad u_-\in H^1(\Omega^\delta_-\cap \Omega,\CC^4), \quad
u_c\in H^1(\RR^3\setminus \overline{\Omega^\delta},\CC^4),\\
\cP_\tau^- u_+ + \cP_\tau^+ u_-=0 \text{ on } \Sigma,
\end{gather*}
and denote by $b^D$ its restriction to the functions vanishing at $\partial \Omega^\delta$,
then in the sense of forms one has $b^N\le a \le b^D$. Furthermore, one has the representations
\[
b^{N/D}=b_0^{N/D} \oplus b_c^{N/D},
\]
where $b^N_0$ is the sesquilinear form in $L^2(\Omega^\delta,\CC^4)$ given by
\begin{multline*}
b^N_0(u,u)=\iiint_{\Omega^\delta \setminus \Sigma} \big| \nabla u \big|^2\, \dd x  
+ \frac{2 m}{\tau} \iint_\Sigma | u_+ - u_-|^2 \dd\Sigma\\
+ \iint_\Sigma M |u_+|^2 \dd\Sigma - \iint_\Sigma M |u_-|^2 \dd\Sigma,
					\quad u_\pm:=u|_{\Omega^\delta_\pm},
\end{multline*}
on the functions $u$ such that $u_\pm\in H^1(\Omega^\delta_\pm,\CC^4)$ and $\cP_\tau^- u_+ + \cP_\tau^+ u_-=0$ on $\Sigma$,
the sesquilinear form $b^N_c$ is given by
\[
b^N_c(u_c,u_c)=\iiint_{\RR^3\setminus \overline{\Omega^\delta}}
\big| \nabla u_c \big|^2\, \dd x,\quad
u_c\in H^1(\Omega^c,\CC^4),
\]
and the forms $b^D_0$ and $b^D_c$ are the restrictions of $b^N_0$ and $b^N_c$, respectively, on functions
vanishing on $\partial\Omega^\delta\equiv \partial(\RR^3\setminus \overline{\Omega^\delta})$.
Due to the obvious inequalities $b^{N/D}_c\ge 0$
and to the fact that $b^{N/D}_0$ define operators with compact resolvents, one has then 
\begin{equation}
   \label{eq-bab}
E_n(b^N_0)\le E_n(a) \le E_n(b^D_0) \quad \text{ for $n$ with $E_n(a)<0$}
\end{equation}
We are now going to give a lower bound of the form $b^N_0$ and an upper bound of the form $b^D_0$
using the above diffeomorphism $\Phi$.
The metric $G$ on $\Sigma\times(-\delta,\delta)$ induced by $\Phi$
takes the form
\[
G(s,t)=\Tilde g (s,t)+ \dd t^2, \quad \Tilde g(s,t):= g(s)\circ(I_s-t S)^2
\]
where $I_s:T_s \Sigma\to T_s\Sigma$ is the identity map, $S:T_s \Sigma\to T_s \Sigma$
is the Weingarten map, $S=\mathrm{d}\, \nu(s)$, and $g$ is the metric of $\Sigma$ induced
by the embedding in $\RR^3$. The associated volume form on $\Sigma\times(-\delta,\delta)$
is given by
\begin{gather*}
\dd \vol(s,t)=\sqrt{\det G(s,t)}\,\dd s\dd t\equiv
\varphi(s,t) \sqrt{\det g} \,\dd s \dd t\equiv \varphi(s,t) \dd\Sigma(s)\dd t,\\
\varphi(s,t):=\det(I_s-t S)=1-2t M(s)+t^2K(s),
\end{gather*}
and we may assume that $\delta$ is sufficiently small to have $\varphi\ge \frac{1}{2}$.

Let us start by considering the unitary transform
	\[
		U : L^2(\Omega^\delta) \ni u\mapsto  (U u) := u \circ \Phi \in L^2\big(\Sigma \times (-\delta,\delta),\dd\vol\big).
	\]
	Then the standard change of variables with the help of the above expressions for the metric tensor
	show that for $\Tilde u:=U u$ one has in the local coordinates
	\[
	\iiint_{\Omega^\delta_\pm} |\nabla u|^2\dd x = \pm\iint_{\Sigma}\int_0^{\pm\delta} \sum_{j,k=1}^3 G^{jk} \langle \partial_j \Tilde u, \partial_k \Tilde u\rangle  \,\dd\vol(s,t),
	\quad (G^{jk}):=(G_{jk})^{-1}.
	\]
		Therefore, 	if we define the sesquilinear forms $b_1^{N/D}$ in  $L^2\big(\Sigma \times (-\delta,\delta),\dd\vol\big)$ by $b^{N/D}_0(u,u)=b^{N/D}_1(U u, U u)$, then $b_1^N$ is given explicitly
		by
	\begin{equation*}
	\begin{split}
	b^N_1(u,u)&=\iiint_{\Sigma\times(-\delta,\delta)}
	 \sum_{j,k=1}^3 G^{jk} \langle \partial_j u, \partial_k u\rangle \varphi\dd\Sigma\dd t\\
	   &~+\dfrac{2m}{\tau}\iint_\Sigma \big| u(\cdot,0^+)-u(\cdot,0^-)\big|^2\dd\Sigma +\iint_\Sigma M\big| u(\cdot,0^+)\big|^2\dd\Sigma - \iint_\Sigma M\big| u(\cdot,0^-)\big|^2\dd\Sigma 
	\end{split}
	\end{equation*}
	on the domain
	\begin{multline*}
	\dom(b^N_1)=U\dom(b^N_0)=\Big\{ u\in H^1\big( \big(\Sigma\times(-\delta,0), \mathbb{C}^4\big)
	\cup \big(\Sigma\times(0,\delta), \mathbb{C}^4\big), 
	\varphi\dd\Sigma \dd t\big):	\\
	\cP_\tau^- u(\cdot,0^+)+\cP_\tau^+ u(\cdot,0^-)=0\Big\}
	\end{multline*}
	and $b^D_1$ is its restriction to the functions vanishing at $\Sigma\times\{\pm\delta\}$. By construction one has $E_n(b^{N/D}_1)=E_n(b^{N/D}_0)$
	for all $n$, and due to \eqref{eq-bab} there holds
	\begin{equation}
   \label{eq-bab2}
E_n(b^N_1)\le E_n(a) \le E_n(b^D_1) \text{ for any $n$ with $E_n(a)<0$.}
\end{equation}

Due to the above expression for $\Tilde g$ one can estimate, with some $C>0$ that for all for $u\in\dom(b_1^{N/D})$ we have 
	\begin{equation*}
      \begin{split}
		(1-C\delta) \sum_{j,k=1}^2 g^{jk} \langle \partial_j u, \partial_k u\rangle
	+|\partial_t u|^2 &\le \sum_{j,k=1}^3 G^{jk} \langle \partial_j u, \partial_k u\rangle\\
	&\le
	(1+C\delta) \sum_{j,k=1}^2 g^{jk} \langle \partial_j u, \partial_k u\rangle
	+|\partial_t u|^2,
    \end{split}
\end{equation*}
	and then $b^N_2\le b^N_1$ and $b^D_1\le b^D_2$, where the form $b^N_2$ is given by
	\begin{equation*}
	\begin{split}
	b^N_2&(u,u)=\iiint_{\Sigma\times(-\delta,\delta)}
	\Big((1-C\delta)\| \nabla_s u\|^2_{T_s\Sigma\otimes\CC^4} + |\partial_t  u|^2\Big) \varphi\, \dd\Sigma\dd t\\
	  &+\dfrac{2m}{\tau}\iint_\Sigma \big| u(\cdot,0^+)-u(\cdot,0^-)\big|^2\dd\Sigma
		+\iint_\Sigma M\big| u(\cdot,0^+)\big|^2\dd\Sigma - \iint_\Sigma M\big| u(\cdot,0^-)\big|^2\dd\Sigma 
	\end{split}
	\end{equation*}
	on the domain $\dom(b^N_2)=\dom(b^N_1)$, and the form $b^D_2$ is given by 
	\begin{equation*}
	\begin{split}
	b^D_2&(u,u)=\iiint_{\Sigma\times(-\delta,\delta)}
	\Big((1+C\delta)\| \nabla_s u\|^2_{T_s\Sigma\otimes\CC^4} + |\partial_t  u|^2\Big) \varphi\, \dd\Sigma\dd t\\
	 &+\dfrac{2m}{\tau}\iint_\Sigma \big| u(\cdot,0^+)-u(\cdot,0^-)\big|^2\dd\Sigma
		+\iint_\Sigma M\big| u(\cdot,0^+)\big|^2\dd\Sigma - \iint_\Sigma M\big| u(\cdot,0^-)\big|^2\dd\Sigma 
	\end{split}
	\end{equation*}
	on the domain $	\dom(b^D_2)=\big\{u\in \dom(b^N_2): u(\cdot,\pm\delta)=0\big\}$.
	Then for any $n$ one has $E_n(b^N_2)\le E_n(b^N_1)$ and $E_n(b^D_1)\le E_n(b^D_2)$ and, due to \eqref{eq-bab2},
		\begin{equation}
   \label{eq-bab3}
E_n(b^N_2)\le E_n(a) \le E_n(b^D_2) \text{ for any $n$ with $E_n(a)<0$}.
\end{equation}

In order to remove the weight $\varphi$ in the above expressions, let us introduce the unitary transform
	\[
	V : L^2\big(\Sigma\times(-\delta,\delta),\varphi \dd\Sigma \dd t)\to L^2\big(\Sigma\times(-\delta,\delta)\big),
	\quad (Vu)(s,t) := \varphi(s,t)^{\frac{1}{2}} u(s,t)
	\]
	and the sesquilinear forms $b^{N/D}_3(u,u)=b^{N/D}_2(V^{-1}u,V^{-1}u)$ defined on $\dom(b^{N/D}_3)=V\big( \dom(b^{D/N}_2)\big)$.
	One sees easily that $\dom(b^{N/D}_3)=\dom(q^{N/D}_{m,\tau,\delta})$. Furthermore,
	to have a more explicit expression for $b_3^{N/D}$ we remark that for
	\begin{equation*}
	v(s,t):=V^{-1}u(s,t) = \varphi(s,t)^{-\frac12}u(s,t)
	\end{equation*}
	one has
	\begin{equation*}
	\partial_t v = \varphi^{-\frac{1}{2}}\partial_t u - \frac{1}{2}\partial_t \varphi \cdot  \varphi^{-\frac{3}{2}} u
	=\varphi^{-\frac{1}{2}}\partial_t u +(M-tK)  \varphi^{-\frac{3}{2}} u.
	\end{equation*}
Hence, we get
	\begin{align*}
		|\partial_t v|^2 & = \varphi^{-1}|\partial_t u|^2 + \varphi^{-3}(M - t K)^2 |u|^2 + \varphi^{-2}(M- tK) \cdot 2\Re \langle\partial_t u,u\rangle\\
		&=\varphi^{-1}|\partial_t u|^2 + \varphi^{-3}(M - t K)^2 |u|^2 + \varphi^{-2}(M- tK) \partial_t \big(|u|^2\big),
	\end{align*}
which implies 
	\begin{multline*}
		\int_{-\delta}^\delta |\partial_t v|^2\varphi \dd t = 
		\int_{-\delta}^\delta |\partial_t u|^2 \dd t
		+\int_{-\delta}^\delta \varphi^{-2} (M-tK)^2	|u|^2\dd t\\
		+ \int_{-\delta}^\delta  \varphi^{-1}(M- tK) \partial_t \big(|u|^2\big)\dd t=:J_1+J_2+J_3.
	\end{multline*}
Using the integration by parts on $(-\delta,0)$ and $(0,\delta)$ we get
\begin{multline*}
J_3=-\int_{-\delta}^{\delta} \partial_t \big(\varphi^{-1}(M- tK)\big)|u|^2\dd t  +\frac{M-\delta K}{1-2\delta M+\delta^2K} \big|u(\cdot,\delta)\big|^2- M \big|u(\cdot,0^+)\big|^2\\
+M \big|u(\cdot,0^-)\big|^2- \frac{M+\delta K}{1+2\delta M + \delta^2K}\big|u(\cdot,-\delta)\big|^2 .
\end{multline*}
In view of the expression for $\varphi$ one sees that uniformly on $\Sigma$ when $\delta$ tends to $0$,
one has
\begin{gather*}
\frac{M\pm\delta K}{1\pm2\delta M+\delta^2K}=M + \cO(\delta), \quad ,\varphi^{-2} (M-tK)^2=M^2+\cO(\delta)\\
-\partial_t \big(\varphi^{-1}(M- tK)\big)=-2M^2+K +\cO(\delta).
\end{gather*}
Therefore, for $u\in \dom(b^{N}_3)$ we can estimate, with a suitable $C>0$,
\begin{multline}
  \label{eq-au1}
\int_{-\delta}^\delta |\partial_t V^{-1} u|^2 \varphi\dd t
\ge \int_{-\delta}^\delta \big( |\partial_t u|^2 
+(K-M^2 -C\delta) |u|^2 \big) \dd t \\
+ M \big|u(\cdot,0^-)\big|^2 - M \big|u(\cdot,0^+)\big|^2
-C \big|u(\cdot,-\delta)\big|^2  -C \big|u(\cdot,\delta)\big|^2,
\end{multline}
while for $u\in \dom(b^{D}_3)$ the terms with $u(\cdot,\pm\delta)$ vanish, thus,
\begin{multline}
  \label{eq-au2}
\int_{-\delta}^\delta |\partial_t V^{-1} u|^2\varphi 
\le \int_{-\delta}^\delta \big(|\partial_t u|^2\dd t
+(K-M^2 +C\delta) |u|^2 \big) \dd t \\
+ M \big|u(\cdot,0^-)\big|^2 - M \big|u(\cdot,0^+)\big|^2.
\end{multline}

In order to control the integral of $ \| \nabla_s v\|^2_{T_s\Sigma\otimes\CC^4} \varphi$
we remark that due to the Cauchy-Schwarz inequality one has
\begin{equation*}
\begin{split}
\| \nabla_s v\|^2_{T_s\Sigma\otimes\CC^4}
&=\Big\| \varphi^{-\frac12}\nabla_s u-\dfrac{1}{2} \varphi^{-\frac32}\nabla_s \varphi \cdot u\Big\|^2_{T_s\Sigma\otimes\CC^4}\\
&\ge (1-\delta)\big\| \varphi^{-\frac12}\nabla_s u\big\|^2_{T_s\Sigma\otimes\CC^4} +\Big(1-\dfrac{1}{\delta}\Big) \Big\|\dfrac{1}{2} \varphi^{-\frac32}\nabla_s \varphi \cdot u\Big\|^2_{T_s\Sigma\otimes\CC^4}\\
&\ge (1-\delta)\varphi^{-1} \|\nabla_s u\|^2_{T_s\Sigma\otimes\CC^4}- \dfrac{t^2}{\delta \varphi^3} \left\|\nabla_s M-\frac{t}{2} \nabla_s K\right\|^2_{T_s\Sigma} |u|^2,
\end{split}
\end{equation*}
which results in	
\begin{multline}
    \label{eq-au3}
\iiint_{\Sigma\times(-\delta,\delta)}  \| \nabla_s v\|^2_{T_s\Sigma\otimes\CC^4} \varphi \dd\Sigma \dd t\\
\ge (1-\delta)\iiint_{\Sigma\times(-\delta,\delta)} \| \nabla_s u\|^2_{T_s\Sigma\otimes\CC^4} \dd\Sigma \dd t
-C'\delta \iiint_{\Sigma\times(-\delta,\delta)}  | u|^2 \dd\Sigma \dd t
\end{multline}
with a suitable $C'>0$. Analogous estimates give
\begin{multline}
   \label{eq-au4}
\iiint_{\Sigma\times(-\delta,\delta)}  \| \nabla_s v\|^2_{T_s\Sigma\otimes\CC^4} \varphi \dd\Sigma \dd t\\
\le (1+\delta)\iiint_{\Sigma\times(-\delta,\delta)} \| \nabla_s u\|^2_{T_s\Sigma\otimes\CC^4} \dd\Sigma \dd t
+C''\delta \iiint_{\Sigma\times(-\delta,\delta)} |u|^2 \dd\Sigma \dd t
\end{multline}
with some $C''>0$. The substitution of \eqref{eq-au1} and \eqref{eq-au3} into the expression for $b^N_3$
and of \eqref{eq-au2} and \eqref{eq-au4} into the expression for $b^D_3$ give the result.
\end{proof}

For the rest of the section we always assume that $\delta$ is any function of $m$ satisfying \eqref{eq-mdelta},
then the assumptions of Lemma~\ref{lem-qnd} are satisfied for large $m$. Recall that the parameter $\mu\in(0,1)$ was introduced in \eqref{eq-aaa}.

\subsection{Upper bound}\label{sec-upp1}

In this section we derive an upper bound for the eigenvalues of $q_{m, \tau, \delta}^D$ from Lemma~\ref{lem-qnd}.
Let us start with the analysis of an auxiliary one-dimensional operator.

\begin{lem} \label{lem1dd}
  Let $s\in\Sigma$, $m>0$, consider the following sesquilinear form $t^D_{s,m}$
	in $L^2\big((-\delta,\delta),\CC^4\big)$ given by
	\begin{align*}
	t^D_{s,m}(u,u)&=\int_{-\delta}^\delta |u'|^2 \dd t  + \frac{2m}{\tau} \big|u(0^+)-u(0^-)\big|^2,\\
	\dom(t^D_{s,m})&=\Big\{ u\in H^1\big((-\delta,\delta)\setminus\{0\},\CC^4\big):  \cP^-_\tau(s) u(0^+)+\cP^+_\tau(s) u(0^-)=u(\pm\delta)=0\Big\},
	\end{align*}
	and let $T^D_{s,m}$ be the associated self-adjoint operator in $L^2\big((-\delta,\delta),\CC^4\big)$.
  Then for $m\to +\infty$ the first eigenvalue of $T^D_{s,m}$ is independent of $s$,
	has the multiplicity $4$ and is given by
	\begin{equation}
	E_1(T^D_{s,m})= -\mu^2m^2 \big(1+\cO(e^{-2\mu m\delta})\big).  \label{eq-e1d}
	\end{equation}
	Furthermore, one can represent, with a suitable smooth function $\psi_m:(0,\delta)\to\RR$
	independent of $s$,
	\begin{multline}
	   \label{eq-vv1}
	\ker\big( T^D_{s,m} -E_1(T^D_{s,m})\big)
	=\Big\{v:
	v(t)=v_\pm \psi_m(|t|) \text{ for } \pm t>0\\ \text{ with }
	v_\pm\in\CC^4 \text{ such that } \cP^-_\tau(s) v_+ +\cP^+_\tau(s) v_-=0
	\Big\}.
	\end{multline}
\end{lem}

\begin{proof}
Let us start by giving a more precise description of $T^D_{s,m}$. 
It is standard to see that $\dom(T^D_{s,m})\subset H^2\big((-\delta,\delta)\setminus\{0\}\big)$
and that $T^D_{s,m}$ acts as $T^D_{s,m} u=-u''$. Therefore, it is sufficient to specify the boundary conditions
at $0$ and $\pm\delta$. Let $v\in \dom(T^D_{s,m})$,
then $v$ belongs to $\dom(t^D_{s,m})$, i.e.
\begin{align}
   \label{eq-p1d}
v(0^+)&=\cR^+_\tau v(0^-),\\
\label{eq-p2d}
v(\pm\delta)&=0,
\end{align}
and $t^D_{s,m}(u,v)=\langle u,T^D_{s,m}v\rangle_{L^2((-\delta,\delta),\CC^4)}$ for all $u\in  \dom (t^D_{s,m})$. 
Using integration by parts on $(-\delta,0)$ and $(0,\delta)$ we conclude that
\begin{align*}
t^D_{s,m}(u,v)&=\int_{-\delta}^\delta \langle u, -v''\rangle\dd t +s^D_{s,m}(u,v),\\
s^D_{s,m}(u,v)&=\big\langle u(0^-),v'(0^-)\big\rangle
-\big\langle u(0^+),v'(0^+)\big\rangle + \frac{2m}{\tau} \big \langle u(0^+)-u(0^-), v(0^+)-v(0^-)\big\rangle.
\end{align*}
Therefore, it is sufficient to check for which $v$ one has $s^D_{s,m}(u,v)=0$ for all $u\in \dom(t^D_{s,m})$.
Due to the fact that $u(0^+) = \cR_\tau^+ u(0^-)$, $u(0^-)\in\CC^4$ is arbitrary and to
\[
s^D_{s,m}(u,v)
=\big\langle u(0^-),v'(0^-)-\cR^+_\tau v'(0^+)\big\rangle 
+ \frac{2m}{\tau} \Big \langle u(0^-), (\cR^+_\tau-I_4)\big(v(0^+)-v(0^-)\big)\Big\rangle,
\]
we have then
\begin{equation}
   \label{eq-p3d}
\cR^+_\tau v'(0^+)-v'(0^-)=\frac{2m}{\tau} (\cR^+_\tau-I_4)\big(v(0^+)-v(0^-)\big).
\end{equation}
Therefore, the domain of $T^D_{s,m}$ consists of the functions $v\in H^2\big((-\delta,\delta)\setminus\{0\}\big)$
satisfying the boundary conditions \eqref{eq-p1d}, \eqref{eq-p2d} and \eqref{eq-p3d}.

One then concludes that a negative number $E=-k^2$ is an eigenvalue of $T^D_{s,m}$ iff one can find
$a_\pm, b_\pm\in \CC^4$, not all zero, such that that the function $v$ given by
\begin{equation*} 
v(t)=\begin{cases}
a_+ e^{-kt} + b_+ e^{kt}, & t>0,\\
a_- e^{kt} + b_- e^{-kt}, & t<0,
\end{cases}
\end{equation*}
satisfies the above boundary conditions. From \eqref{eq-p2d} we deduce
\begin{equation*}
   a_\pm = \theta b_\pm, \quad \theta:=-e^{2k\delta}
\end{equation*}
and hence
\begin{equation*}
v(t)=\begin{cases}
(\theta e^{-kt} +  e^{kt})b_+ , & t>0,\\
(\theta e^{kt}  +   e^{-kt})b_- , & t<0.
\end{cases}
\end{equation*}
It follows then from \eqref{eq-p1d} that $b_+=\cR^+_\tau b_-$ and
\begin{equation} \label{equation_eigenfunction_upper_bound}
v(t)=\begin{cases}
(\theta e^{-kt} +  e^{kt})\cR^+_\tau b_- , & t>0,\\
(\theta e^{kt}  +   e^{-kt})b_- , & t<0.
\end{cases}
\end{equation}
Then
\begin{align*}
v(0^+)&=(\theta+1)\cR^+_\tau b_-,& v(0^-)&=(\theta+1)b_-,\\
v'(0^+)&=-k(\theta-1)\cR^+_\tau b_-, & v'(0^-)&=k(\theta-1)b_-,
\end{align*}
and the substitution into \eqref{eq-p3d} shows that $E=-k^2$
is an eigenvalue iff the equation
\[
-k(\theta-1)\big((\cR^+_\tau)^2+I_4\big)b_-=\dfrac{2m}{\tau}(\theta+1)(\cR^+_\tau -I_4)^2 \ b_-
\]
admits a solution $b_-\ne 0$. A straightforward calculation shows that
\begin{equation*}
  (\cR_\tau^+)^2 + I_4 = \frac{2 (\tau^2 + 4)}{4 - \tau^2} \cR_\tau^+ \quad \text{and} \quad
  (\cR_\tau^+ - I_4)^2 = \tau \cdot \frac{4 \tau}{4 - \tau^2} \cR_\tau^+
\end{equation*}
and therefore, one may rewrite the last condition as
\[
k \dfrac{\theta-1}{\theta+1}\,b_-=-\dfrac{2m}{\tau}  \cdot \big((\cR^+_\tau)^2+I_4\big)^{-1}(\cR^+_\tau-I_4)^2 \ b_- = \mu m b_-
\]
with $\mu$ given by \eqref{eq-aaa}.
Therefore, a solution $b_-\ne 0$ to the preceding equation exists iff  $k$ satisfies
\begin{equation}
   \label{eq-kk1d}
k\dfrac{\theta-1}{\theta+1}=\mu m,
\end{equation}
and in that case the first eigenvalue is four times degenerate due to the arbitrary choice of $b_-\in\CC^4$,
and the representation \eqref{eq-vv1} follows from the preceding representation for $v$
in~\eqref{equation_eigenfunction_upper_bound}.
In order to show the uniqueness of the lowest eigenvalue and to study the behavior with respect to $m$ and $\delta$, let us rewrite
\eqref{eq-kk1d} in the form
\[
F(k\delta)=\mu m \delta, \quad F(x)= x \coth x.
\]
Since
\begin{equation*}
  F'(x) = \frac{\sinh x \cosh x - x}{\sinh^2 x} > 0, \quad x > 0,
\end{equation*}
one remarks that $F:(0,+\infty)\to(1,+\infty)$ is a diffeomorphism, with $F(0^+)=1$ and $F(+\infty)=+\infty$,
which shows that the solution $k$ is unique for  $\mu m\delta>1$. Furthermore, for $m\delta\to+\infty$
one has obviously $k\delta\to+\infty$, which implies that $\theta\to -\infty$. The substitution into
\eqref{eq-kk1d} shows that $k\sim \mu m$, and another use of \eqref{eq-kk1d} gives \eqref{eq-e1d}.
\end{proof}

Now we are going to use the preceding lemma in order to establish an upper estimate
for the eigenvalues defined by the form $q^D_{m,\tau,\delta}$ from Lemma~\ref{lem-qnd}:
\begin{lem}\label{lem-upb}
There exists $C>0$ and $m_0>0$ such that for $m>m_0$ and any $j\in\NN$ there holds
\[
E_j(q^D_{m,\tau,\delta})\le -\Big( \dfrac{4 \tau}{\tau^2+4}\Big)^2 m^2 + E_j\Big( (1+C\delta)\cL^\tau_0 + K-M^2\Big)
+C\delta+ C m^2e^{-2\mu m\delta}.
\]
\end{lem}

\begin{proof}
Recall that $\dom(\ell^\tau_0)$ is defined in \eqref{eq-ll1}. Define for $v = (v_+, v_-) \in\dom(\ell^\tau_0)$
\[
u_v(s,t)=
c_m\psi_m\big(|t|\big)\begin{cases}
v_+(s), & t>0,\\
v_-(s), & t<0,
\end{cases}
\]
with the function $\psi_m$ as in \eqref{eq-vv1}, where the constant $c_m>0$
is independent of $s$ and chosen by $c_m^2\|\psi_m\|^2_{L^2(0,\delta)}=1$.
Then $u_v\in \dom(q^D_{m,\tau,\delta})$ with $\|u_v\|_{L^2(\Sigma\times(-\delta,\delta),\CC^4)}=\|v\|_\cH$.
Due to the choice of $\psi_m$ and $v$ one has then
\begin{multline*}
\iint_\Sigma \bigg(\int_{-\delta}^\delta |\partial_t u_v|^2 \dd t  + \frac{2m}{\tau} \big|u_v(\cdot,0^+)-u_v(\cdot,0^-)\big|^2 \bigg) \dd\Sigma\\
=E_1(T^D_{s,m}) \|u_v\|^2_{L^2(\Sigma\times(-\delta,\delta),\CC^4)}=E_1(T^D_{s,m})\|v\|^2_\cH
\end{multline*}
with the operator $T^D_{s,m}$ from Lemma~\ref{lem1dd}. One also has
\begin{multline*}
\iiint_{\Sigma\times (-\delta,\delta)} \Big((1+c\delta)\|\nabla_s u_v\|^2_{T_s\Sigma\otimes\CC^4}
+ (K-M^2+c\delta)\,|u_v|^2\Big)\dd\Sigma\dd t\\
\begin{aligned}
&=\int_0^\delta c_m^2 \big|\psi_m(t)\big|^2 \dd t \cdot 
\iint_\Sigma \Big((1+c\delta)  \|\nabla_s v_+\|^2_{T_s\Sigma\otimes\CC^4}
+ (K-M^2+c\delta)\,|v_+|^2\Big)\dd\Sigma\\
&\quad +\int_{-\delta}^0 c_m^2 \big|\psi_m(-t)\big|^2 \dd t \cdot 
\iint_\Sigma \Big((1+c\delta) \|\nabla_s v_-\|^2_{T_s\Sigma\otimes\CC^4}
+ (K-M^2+c\delta)\,|v_-|^2\Big)\dd\Sigma\\
&=\iint_\Sigma \Big((1+c\delta) \|\nabla_s v_+\|^2_{T_s\Sigma\otimes\CC^4}
+ (K-M^2+c\delta)\,|v_+|^2\Big)\dd\Sigma\\
&\quad+\iint_\Sigma \Big((1+c\delta) \|\nabla_s v_-\|^2_{T_s\Sigma\otimes\CC^4}
+ (K-M^2+c\delta)\,|v_-|^2\Big)\dd\Sigma\\
&=(1+c\delta)\ell^\tau_0(v,v)+ \big\langle v,(K-M^2+c\delta) v\big\rangle_\cH,
\end{aligned}
\end{multline*}
i.e.
\[
q^D_{m,\tau,\delta}(u_v,u_v)=(1+c\delta)\ell^\tau_0(v,v)+ \big\langle v,(K-M^2+c\delta) v\big\rangle_\cH + E_1(T^D_{s,m})\|v\|^2_\cH.
\]
By Lemma~\ref{lem1dd} one can find $m_0>0$ and $C>0$ independent of $s$ such that
\[
E_1(T^D_{s,m,\delta})\le -\mu^2m^2+Cm^2 e^{-2\mu m\delta} \text{ for } m>m_0,
\]
and then
$$
\dfrac{q^D_{m,\tau,\delta}(u_v,u_v)}{\|u_v\|_{L^2(\Sigma\times(-\delta,\delta),\CC^4)}}
\le -\mu^2m^2 +\dfrac{(1+c\delta)\ell^\tau_0(v,v)+ \big\langle v,(K-M^2) v\big\rangle_\cH}{\|v\|^2_\cH} + Cm^2 e^{-2\mu m\delta} + c\delta.
$$
If $F_j$ is a $j$-dimensional subspace of $\dom(\ell^\tau_0)$, then $\cF_j:=\{u_v: v\in F_j\}$
is a $j$-dimensional subspace of $\dom(q^D_{m,\tau,\delta})$,
and by the min-max-principle one has, by estimating all constants by a generic constant $C$,
\begin{equation*}
\begin{split}
E_j&(q^D_{m,\tau,\delta})\le \min_{\cF_j}\max_{u\in \cF_j} \dfrac{q^D_{m,\tau,\delta}(u,u)}{\|u\|_{L^2(\Sigma\times(-\delta,\delta),\CC^4)}}
\leq\min_{F_j}\max_{v\in F_j} \dfrac{q^D_{m,\tau,\delta}(u_v,u_v)}{\|u_v\|_{L^2(\Sigma\times(-\delta,\delta),\CC^4)}}\\
&\le -\mu^2m^2 +\min_{F_j}\max_{v\in F_j} \dfrac{(1+C\delta)\ell^\tau_0(v,v)+ \big\langle v,(K-M^2) v\big\rangle_\cH}{\|v\|^2_\cH} 
+ Cm^2 e^{-2\mu m\delta} + C\delta\\
&=-\mu^2m^2 +E_j\Big((1+C\delta)\cL^\tau_0+K-M^2\Big) + Cm^2 e^{-2\mu m\delta} + C\delta. \qedhere
\end{split}
\end{equation*}
\end{proof}

\begin{proof}[Proof of the upper bound in Proposition~\ref{prop-main1}]
It is sufficient to substitute the estimate of Lemma~\ref{lem-upb}
into the upper bound of Lemma~\ref{lem-qnd} and to use
\begin{equation*}
m^2-\mu^2m^2=(1-\mu^2)m^2=\Big(\dfrac{\tau^2-4}{\tau^2+4}\Big)^2 m^2. \qedhere
\end{equation*}
\end{proof}

\subsection{Lower bound}\label{sec-low1}

We start with an estimate for another auxiliary one-dimensional operator.

\begin{lem} \label{lem1dn}
  For $m, c>0$  let $h_{m,c}$ be the sesquilinear form in $L^2\big((-\delta,\delta),\CC^4\big)$ given by
	\begin{align*}
	h_{m,c}(u,u)&=\int_{-\delta}^\delta |u'|^2 \dd t  + \frac{2m}{\tau} \big|u(0^+)-u(0^-)\big|^2 -c \big|u(\delta)\big|^2-c\big|u(-\delta)\big|^2,\\
	\dom(h_{m,c})&=\Big\{ u\in H^1\big((-\delta,\delta)\setminus\{0\},\CC^4\big): \widetilde{P}^-_\tau u(0^+)+\widetilde{P}^+_\tau u(0^-)=0\Big\},\\
	\widetilde{P}^\pm_\tau&:=\dfrac{\tau}{2}\pm \beta,
	\end{align*}
	and let $H_{m,c}$ be the associated self-adjoint operator in $L^2\big((-\delta,\delta),\CC^4\big)$.
  Then for $m\to +\infty$ the first eigenvalue has the multiplicity $4$ and
	\begin{gather}
	E_1(H_{m,c})= -\mu^2m^2 \big(1+\cO(e^{-2\mu m\delta})\big), \label{eq-e1}\\
	E_5(H_{m,c})\ge \dfrac{b^2}{\delta^2} \quad \text{ for some } b>0. \label{eq-e5}
	\end{gather}
	Furthermore, one can represent, with a suitable smooth function $\psi_{m,c}:(0,\delta)\to\RR$,
	\begin{multline}
	   \label{eq-vv2}
	\ker\big( H_{m,c} -E_1(H_{m,c})\big)
	=\Big\{u:
	u(t)=v_\pm \psi_{m,c}\big(|t|\big) \text{ as } \pm t>0\\ \text{ with }
	v_\pm\in\CC^4 \text{ such that } \widetilde{P}^-_\tau v_+ +\widetilde{P}^+_\tau v_-=0
	\Big\}.
	\end{multline}
\end{lem}

\begin{proof}
In the proof we rewrite the condition $\widetilde{P}^-_\tau u(0^+)+\widetilde{P}^+_\tau u(0^-)=0$
as $u(0^+)=\widetilde{R}^+_\tau u(0^-)$ with $\widetilde{R}^+_\tau=-(\widetilde{P}^-_\tau)^{-1}\widetilde{P}^+_\tau$.

 Let us give first
a more precise description of $H_{m,c}$. It is standard to see that 
$\dom(H_{m,c})\subset H^2\big((-\delta,\delta)\setminus\{0\},\CC^4\big)$
and that $H_{m,c}$ acts as $H_{m,c} u=-u''$. Therefore, it is sufficient to specify the boundary conditions
at $0$ and $\pm\delta$. Let $v\in \dom(H_{m,c})$,
then $v$ belongs to $\dom(h_{m,c})$, i.e.
\begin{equation}
   \label{eq-p1}
v(0^+)=\widetilde{R}^+_\tau v(0^-),
\end{equation}
and $h_{m,c}(u,v)=\langle u,H_{m,c} v\rangle_{L^2}$ for all $u\in  \dom (h_{m,c})$. 
Using integration by parts on $(-\delta,0)$ and $(0,\delta)$ we conclude that
\begin{align*}
h_{m,c}(u,v)&=\int_{-\delta}^\delta \langle u, -v''\rangle\dd t +s_{m, c}(u,v),\\
s_{m,c}(u,v)&=-\big\langle u(-\delta), v'(-\delta)\big\rangle+ \big\langle u(0^-),v'(0^-)\big\rangle
-\big\langle u(0^+),v'(0^+)\big\rangle\\
&\qquad+\big\langle u(\delta),v'(\delta)\big\rangle
+ \frac{2m}{\tau} \big \langle u(0^+)-u(0^-), v(0^+)-v(0^-)\big\rangle\\
&\qquad-c \big\langle u(\delta),v(\delta)\big\rangle-c \big\langle u(-\delta),v(-\delta)\big\rangle.
\end{align*}
Therefore, it is sufficient to check for which $v$ one has $s_{m,c}(u,v)=0$ for all $u\in \dom(h_{m,c})$.
Testing on $u$ localized near $\pm\delta$ one concludes that $v$ must satisfy
\begin{equation}
   \label{eq-p2}
v'(\pm\delta)=\pm c v(\pm\delta).
\end{equation}
Now assume that $u$ vanishes at $\pm \delta$, then using
\eqref{eq-p1} one rewrites
\begin{align*}
s_{m,c}(u,v)&
=\big\langle u(0^-),v'(0^-)\big\rangle -\big\langle u(0^+),v'(0^+)\big\rangle \\
&\qquad + \frac{2m}{\tau} \big \langle u(0^+)-u(0^-), v(0^+)-v(0^-)\big\rangle\\
&=\big\langle u(0^-),v'(0^-)-\widetilde{R}^+_\tau v'(0^+)\big\rangle
+ \frac{2m}{\tau} \Big \langle u(0^-), (\widetilde{R}^+_\tau-I_4)\big(v(0^+)-v(0^-)\big)\Big\rangle
\end{align*}
implying 
\begin{equation}
   \label{eq-p3}
\widetilde{R}^+_\tau v'(0^+)-v'(0^-)=\frac{2m}{\tau} (\widetilde{R}^+_\tau-I_4)\big(v(0^+)-v(0^-)\big).
\end{equation}
By summarizing the above, the domain of $H_{m,c}$ consists of the functions $v\in H^2\big((-\delta,\delta)\setminus\{0\},\CC^4\big)$
satisfying the boundary conditions \eqref{eq-p1}, \eqref{eq-p2} and \eqref{eq-p3}.

One then concludes that a negative number $E=-k^2$ with $k>0$ is an eigenvalue of $H_{m,c}$ iff one can find
$a_\pm, b_\pm\in \CC^4$, not all zero, such that the associated eigenfunction 
\[
v(t)=\begin{cases}
a_+ e^{-kt} + b_+ e^{kt}, & t>0,\\
a_- e^{kt} + b_- e^{-kt}, & t<0,
\end{cases}
\]
satisfies the above boundary conditions. From \eqref{eq-p2} we deduce
\begin{equation*}
    a_\pm = \theta b_\pm, \quad \theta:=\dfrac{k-c}{k+c}e^{2k\delta},
\quad \text{i.e.} \quad 
v(t)=\begin{cases}
(\theta e^{-kt} +  e^{kt})b_+ , & t>0,\\
(\theta e^{kt}  +   e^{-kt})b_- , & t<0.
\end{cases}
\end{equation*}
It follows then from \eqref{eq-p1} that $b^+=\widetilde{R}^+_\tau b_-$ and
\[
v(t)=\begin{cases}
(\theta e^{-kt} +  e^{kt})\widetilde{R}^+_\tau b_- , & t>0,\\
(\theta e^{kt}  +   e^{-kt})b_- , & t<0.
\end{cases}
\]
Then
\begin{align*}
v(0^+)&=(\theta+1)\widetilde{R}^+_\tau b_-,& v(0^-)&=(\theta+1)b_-,\\
v'(0^+)&=-k(\theta-1)\widetilde{R}^+_\tau b_-, & v'(0^-)&=k(\theta-1)b_-,
\end{align*}
and the substitution into \eqref{eq-p3} shows that $E=-k^2$
is an eigenvalue iff the equation
\[
-k(\theta-1)\big((\widetilde{R}^+_\tau)^2+I_4\big)b_-=\dfrac{2m}{\tau}(\theta+1)(\widetilde{R}^+_\tau-I_4)^2 \ b_-
\]
admits a solution $b_-\ne 0$. One may rewrite the last condition as
\[
k \dfrac{\theta-1}{\theta+1}\,b_-=-\dfrac{2m}{\tau}  \big((\widetilde{R}^+_\tau)^2+I_4\big)^{-1}\cdot (\widetilde{R}^+_\tau-I_4)^2 b_-,
\]
and using the equality $\beta^2=I_4$ we compute
\begin{equation}
   \label{eq-err1}
\big((\widetilde{R}^+_\tau)^2+I_4\big)^{-1}\cdot (\widetilde{R}^+_\tau-I_4)^2 =\dfrac{2\tau^2}{4+\tau^2} I_4.
\end{equation}
Therefore, a solution $b_-\ne 0$ to the above equation exists iff  $k$ satisfies
\begin{equation}
   \label{eq-kk1}
k\dfrac{\theta-1}{\theta+1}=m \mu
\end{equation}
with $\mu$ given by \eqref{eq-aaa},
and in that case the first eigenvalue is four times degenerate due to the arbitrary choice of $b_-$,
and the representation \eqref{eq-vv1} follows from the preceding constructions of the function $v$.
In order to show the uniqueness of $k$ and to study its behavior with respect to $m$, let us rewrite \eqref{eq-kk1}
as
\[
F_{c\delta}(k\delta)=\mu m \delta, \quad F_\varepsilon(x):= x \dfrac{x \tanh x - \varepsilon}{x-\varepsilon \tanh x}.
\]
One remarks that for $\varepsilon\in(0,1)$ the function $F_\varepsilon:(0,+\infty)\to \RR$ is well-defined and
\[
F'_\varepsilon (x)=x\dfrac{\varepsilon(1-\tanh^2 x) + \dfrac{x^2-\varepsilon^2}{\cosh^2 x}}{(x-\varepsilon \tanh x)^2} 
+ \dfrac{1}{x} F_\varepsilon(x),
\]
and $F'_\varepsilon(x)>0$ provided  $x>\varepsilon$ and $F_\varepsilon(x)>0$. Furthermore,
$F_\varepsilon(x)>0$ if and only if $x\tanh x>\varepsilon$, therefore, $F^{-1}_\varepsilon\big((0,+\infty)\big)$
is a subinterval of $(\varepsilon,+\infty)$.
It follows that $F_\varepsilon:F^{-1}_\varepsilon\big((0,\infty)\big)\to (0,\infty)$ is a diffeomorphism,
and there exists a unique solution $k$ provided $c\delta<1$, which is satisfied for large $m$ due to \eqref{eq-mdelta}, 
as $\mu m \delta > 0$.
On the other hand, $F_\varepsilon(x)$ is decreasing in $\varepsilon$ due to
\[
\dfrac{\partial F_\varepsilon(x)}{\partial\varepsilon}=-x^2\,\dfrac{1-\tanh^2 x}{(x-\varepsilon \tanh x)^2}<0,
\]
which implies $k\delta\ge k_0\delta$ with $k_0>0$ being the solution to $F_0(k_0\delta)=\mu m\delta$.
As $F_0(x)=x\tanh x$, one easily checks that $k_0\delta\to +\infty$ for $m\delta\to +\infty$, and then $k\delta\to+\infty$
and $\theta\to +\infty$. Therefore,  $k\sim m \mu$ for large $m$ due to \eqref{eq-kk1}, and another iteration of \eqref{eq-kk1} 
gives \eqref{eq-e1}.

In order to estimate the next eigenvalue of $H_{m,c}$ we proceed first in the same way and show that
$E=k^2$ with $k>0$ is an eigenvalue iff
\begin{equation}
   \label{eq-gkd}
G_{c\delta}(k\delta)=\mu m\delta, \quad G_\varepsilon(x):=F_\varepsilon(ix)=-x\dfrac{x\tan x+\varepsilon}{x-\varepsilon\tan x}.
\end{equation}
Using the convexity of $x\mapsto \tan x$ one sees that
$0<\tan x< \frac{4}{\pi} x$ for $x\in\big(0,\frac{\pi}{4}\big)$, hence,
$G_\varepsilon(x)<0$ for all $x\in\big(0,\frac{\pi}{4}\big)$ and $\varepsilon\in\big(0,\frac{4}{\pi}\big)$.
As $\mu m\delta>0$, it follows that \eqref{eq-gkd} admits no solution $k$ with $k\delta\in \big(0,\frac{\pi}{4}\big)$ as $m$ is large,
in other words, the operator $H_{m,c}$ has no eigenvalues in $\big(0, \frac{\pi^2}{16 \delta^2}\big)$ for $m\to +\infty$.
In order to complete the proof of \eqref{eq-e5} it remains to check that $0$ is not an eigenvalue of $H_{m,c}$ for $m\to +\infty$.
If $0$ were an eigenvalue, then there would exist $a_\pm,b_\pm\in \CC^4$, not all zero, for which the function
\[
v(t)=\begin{cases} 
a_+ + b_+t, & t>0,\\
a_- - b_-t, & t<0,
\end{cases}
\]
would satisfy the boundary conditions \eqref{eq-p1}, \eqref{eq-p2}, \eqref{eq-p3}. The condition \eqref{eq-p1} gives
\begin{equation*}
  a_+=\widetilde{R}^+_\tau a_- \quad \text{and} \quad v(0^+)=\widetilde{R}^+_\tau a_-, \quad v(0^-)=a_-,
\end{equation*}
and \eqref{eq-p2} implies
\begin{equation*}
  b_\pm = \dfrac{c}{1-c\delta} a_\pm,
\end{equation*}
hence we deduce
\begin{equation*}
  v(t)=\dfrac{1}{1-c\delta}\begin{cases}
(1-c\delta+ct)\widetilde{R}^+_\tau a_-, & t>0,\\
(1-c\delta-ct)a_-, & t<0.
\end{cases}
\end{equation*}
This yields
\begin{equation*}
  v'(0^+)=\dfrac{c}{1-c\delta}\widetilde{R}^+_\tau a_- \quad \text{and} \quad v'(0^-)=-\dfrac{c}{1-c\delta}a_-
\end{equation*}
and the substitution into \eqref{eq-p3} together with the identity \eqref{eq-err1} imply
\[
\dfrac{c}{1-c\delta}\big((\widetilde{R}^+_\tau)^2+1\big)a_-=\dfrac{2m}{\tau}(\widetilde{R}^+_\tau-1)^2a_-, 
\text{ i.e. } \Big(\dfrac{c}{1-c\delta}+\dfrac{4m|\tau|}{\tau^2+4}\Big)a_-=0.
\]
As the number in the parentheses is non-zero for large $m$, the only solution is the trivial one $a_-=0$,
which then implies that $0$ is not an eigenvalue of $H_{m,c}$.
\end{proof}

For what follows we need a special representation for the matrices $\cB$ from~\eqref{def_M}:

\begin{lem}\label{lem-btheta}
For each $s\in\Sigma$ there holds $\cB(s)=\Theta_0(s) \beta \Theta_0(s)^*$
with the unitary matrices $\Theta_0(s) \in \CC^{4 \times 4}$ given by
\begin{equation}
   \label{eq-theta0}
\Theta_0(s)=\dfrac{1}{\sqrt{2}}\Big( I_4+\rmi\alpha\cdot\nu(s)\Big).
\end{equation}
\end{lem}

\begin{proof}
Using  $(\alpha\cdot \nu)^2=I_4$ one easily checks that  $\Theta_0(s)^*=\frac{1}{\sqrt{2}}\big( I_4-\rmi\alpha\cdot\nu(s)\big)$
and that $\Theta^*_0(s) \Theta_0(s)=I_4$, i.e. that $\Theta_0$ is unitary. 
Moreover, using the commutation relations \eqref{eq_commutation} we have 
\begin{multline*}
\Theta_0 \beta \Theta^*_0=\dfrac{1}{2}( 1+\rmi\alpha\cdot\nu)\beta ( 1-\rmi\alpha\cdot\nu)=\dfrac{1}{2}( 1+\rmi\alpha\cdot\nu)( 1+\rmi\alpha\cdot\nu) \beta\\
=\dfrac{1}{2}\big( I_4+ 2\rmi \alpha\cdot \nu - (\alpha\cdot \nu)^2\big)\beta =\rmi (\alpha\cdot \nu) \beta=-\rmi\beta \alpha\cdot \nu=\cB. \qedhere
\end{multline*}
\end{proof}

An explicit computation with the help of Lemma~\ref{lem-btheta} gives then the following result.
\begin{lem} \label{lem1dn1}
  For $s\in\Sigma$, $m>0$, and $c>0$ consider the following sesquilinear form $t^N_{s,m,c}$ in $L^2\big((-\delta,\delta),\CC^4\big)$:
	\begin{align*}
	t^N_{s,m,c}(u,u)&=\int_{-\delta}^\delta |u'|^2 \dd t  + \frac{2m}{\tau} \big|u(0^+)-u(0^-)\big|^2 -c \big|u(\delta)\big|^2-c\big|u(-\delta)\big|^2,\\
	\dom(t^N_{s,m,c})&=\Big\{ u\in H^1\big((-\delta,0)\cup(0,\delta),\CC^4\big): \cP^-_\tau(s) u(0^+)+\cP^+_\tau(s) u(0^-)=0\Big\},
	\end{align*}
	then the associated self-adjoint operator $T^N_{s,m,c}$ in $L^2\big((-\delta,\delta),\CC^4\big)$ is unitarily equivalent to 
	the operator $H_{m,c}$ from Lemma~\ref{lem1dn},
	\begin{equation}
	 T^N_{s,m,c}=\Theta(s) H_{m,c} \Theta(s)^*,
	\end{equation}
	where $\Theta(s)$ is the unitary map in $L^2\big((-\delta,\delta),\CC^4\big)$ defined by $\big(\Theta(s) u\big) (t):=\Theta_0(s) u(t)$
	with $\Theta_0$ given by \eqref{eq-theta0}, and $s\mapsto \Theta(s)$ is a $C^2$ map in the operator norm topology. Furthermore,
	one can represent, with a suitable smooth function $\psi_{m,c}:(0,\delta)\to\RR$ 	independent of $s$,
	\begin{multline}
	   \label{eq-vv3}
	\ker\big( T^N_{s,m,c} -E_1(T^N_{s,m,c})\big)
	=\Big\{v:
	v(t)=v_\pm \psi_{m,c}(|t|) \text{ as } \pm t>0\\ \text{ with }
	v_\pm\in\CC^4 \text{ such that } \cP^-_\tau(s) v_+ +\cP^+_\tau(s) v_-=0
	\Big\}.
	\end{multline}

\end{lem}

Now let us recall some standard constructions, for which it is useful to use the identification
\[
L^2\big(\Sigma\times(-\delta,\delta),\CC^4)\simeq L^2(\Sigma, \cG), \quad
\cG:=L^2\big((-\delta,\delta),\CC^4).
\]
Recall that for any Banach space $B$ the gradient $\nabla_s:C^1(\Sigma, B)\to C^0(T\Sigma,B)$
acts in local coordinates of $\Sigma$ as
\[
(\nabla_s f)_j=\sum_{k} g^{jk} \partial_k f.
\]
In particular, for the $C^2$ maps $\Theta:\Sigma\to \bfB(\cG)$ and $\Theta^*:\Sigma\to \bfB(\cG)$ 
from Lemma~\ref{lem1dn1} one can find
a constant $C>0$ such that for every $u\in C^0(\Sigma,\cG)$ at every point $s\in\Sigma$ there holds
\begin{equation}
    \label{eq-noth}
\big\|(\nabla_s \Theta) u\big\|_{T_s\Sigma\otimes\cG}\le C \|u\|_{\cG}, \quad
\big\|(\nabla_s \Theta^*) u\big\|_{T_s\Sigma\otimes\cG}\le C \|u\|_{\cG},
\end{equation}
and $C$ is independent of $m$ and $\delta$.
Furthermore, let $\pi(s)$ be the orthogonal projector in $L^2\big((-\delta,\delta),\CC^4\big)$ on the subspace
$\ker\big(T^N_{s,m,c}-E_1(T^N_{s,m,c})\big)$.
Denote by $\Pi$ the orthogonal projector in $L^2\big(\Sigma\times(-\delta,\delta),\CC^4)$
given by
\[
(\Pi u)(s,t)=\pi(s) u(s,\cdot) (t)
\]
and set $\Pi^\perp:=1-\Pi$.
Due to the fibered structure, both $\Pi$ and $\Pi^\perp$ also define in the canonical way bounded operators in $L^2(T\Sigma)\otimes L^2\big((-\delta,\delta),\CC^4)$,
to be denoted by the same symbols.

\begin{lem}\label{lem29}
The map
$[\nabla_s,\Pi] u:=\nabla_s (\Pi u) -\Pi(\nabla_s u)$ defined for $u\in C^1(\Sigma) \otimes L^2\big((-\delta,\delta),\CC^4)$
extends by density to a bounded operator
\[
[\nabla_s,\Pi]: L^2(\Sigma) \otimes L^2\big((-\delta,\delta),\CC^4)\to L^2(T\Sigma) \otimes L^2\big((-\delta,\delta),\CC^4),
\]
whose norm remains uniformly bounded for $\delta\to 0^+$ and $m\delta\to+\infty$. Moreover we have $[\nabla_s,\Pi]\Big(H^1(\Sigma) \otimes L^2\big((-\delta,\delta)\Big) \subset H^1(T\Sigma) \otimes L^2\big((-\delta,\delta),\CC^4)$. The same conclusion holds
for $[\nabla_s,\Pi^\perp]\equiv -[\nabla_s,\Pi]$.
\end{lem}

\begin{proof}
By Lemma \ref{lem-btheta} one can represent $\pi(s)=\Theta(s)\pi_0 \Theta(s)^*$, where 
$\pi_0$ is the orthogonal projector in $L^2\big((-\delta,\delta),\CC^4)$
on $\ker\big(H_{m,c}- E_1(H_{m,c})\big)$ with the operator $H_{m,c}$ from Lemma~\ref{lem1dn}.
As $\pi_0$ does not depend on $s$, a direct computation in the local coordinates shows 
that at each point $s\in\Sigma$ one has
\begin{equation}\label{eqn:expcommut}
[\nabla_s,\Pi] u= (\nabla_s \Theta) \pi_0 \Theta^*u + \Theta \pi_0 (\nabla_s \Theta^*)u.
\end{equation}
Using \eqref{eq-noth} we estimate
\begin{align*}
\big\| (\nabla_s \Theta) \pi_0 \Theta^*u\big\|_{T_s \Sigma \otimes\cG}&\le 
C \|\pi_0 \Theta^*u\|_{\cG} \le C\|\pi_0\|_{\bfB(G)}\| \Theta^*\|_{\bfB(\cG)} \|u\|_{\cG}\le C \|u\|_{\cG},\\
\big\| \Theta \pi_0 (\nabla_s \Theta^*)u\big\|_{T_s \Sigma \otimes\cG}&
\le 
\big\| \Theta\|_{\bfB(T_s \Sigma \otimes\cG)}
\big\|\pi_0\|_{\bfB(T_s \Sigma \otimes\cG)} \|(\nabla_s \Theta^*) u\|_{T_s \Sigma \otimes\cG} \le C \|u\|_{\cG},
\end{align*}
then
\begin{multline*}
\big\|[\nabla_s,\Pi] u\big\|^2_{L^2(T\Sigma)\otimes\cG}
=\iint_\Sigma \big\| (\nabla_s \Theta) \pi_0 \Theta^*u + \Theta \pi_0 (\nabla_s \Theta^*) u\big\|^2_{T_s \Sigma \otimes\cG}\dd\Sigma(s)\\
\le
2\iint_\Sigma \big\| (\nabla_s \Theta) \pi_0 \Theta^*u\big\|_{T_s \Sigma \otimes\cG}^2\dd\Sigma(s)
+
2\iint_\Sigma \big\| \Theta \pi_0 (\nabla_s \Theta^*)u\big\|_{T_s \Sigma \otimes\cG}^2\dd\Sigma(s)\\
\le 4C^2 \iint_\Sigma \| u\|^2_\cG \dd\Sigma(s)=4C^2\|u\|^2_{L^2(\Sigma)\otimes \cG}.
\end{multline*}
As the constant $C$ is independent of $m$ and $\delta$, the continuity result follows. To prove the mapping properties of $[\nabla_s,\Pi]$ between the Sobolev spaces of order $1$,
it is enough to remark that \eqref{eqn:expcommut} is differentiable with respect to $s$ because $\Theta$ is $C^2$.
\end{proof}

\begin{lem}\label{lem-lowb}
Let the form $q^N_{m,\tau,\delta}$ be as in Lemma~\ref{lem-qnd} and let $\mu$ be given by~\eqref{eq-aaa}. Then
there are constants $b>0$ and $m_0>0$ such that for all $m>m_0$ and $j\in\big\{1,\dots,\cN(q^N_{m,\tau,\delta},0)\big\}$ it holds
\[
E_j(q^N_{m,\tau,\delta})\ge -\mu^2m^2+E_j\big((1-b\delta)\cL^\tau_0 +K-M^2\big)- b m^2e^{-2\mu m\delta}-b\delta.
\]
\end{lem}

\begin{proof}
Let $c>0$ be as in the expression for $q^N_{m,\tau,\delta}$.
Then by Lemmas~\ref{lem1dn} and \ref{lem1dn1} one may estimate, with some $b_0, b_1, m_0>0$ independent of $s$,
\begin{equation}
   \label{eq-lambdas}
E_1(T^N_{s,m,c})\ge -\mu^2m^2 - b_0 m^2e^{-2\mu m\delta}, \quad
E_5(T^N_{s,m,c})\ge \dfrac{b_1^2}{\delta^2}
\quad \text{ for } m>m_0.
\end{equation}

Let $u \in \dom(q^N_{m,\tau,\delta})$ be fixed.
Due to the definition of $\Pi$ and with the help of the min-max principle one obtains
\begin{multline*}
\iint_\Sigma \bigg(\int_{-\delta}^\delta |\partial_t u|^2 \dd t  + \frac{2m}{\tau} \big|u(\cdot,0^+)-u(\cdot,0^-)\big|^2 -c \big|u(\cdot,\delta)\big|^2-c\big|u(\cdot,-\delta)\big|^2\bigg) \dd\Sigma\\
\ge E_1(T^N_{s,m,c}) \|\Pi u\|^2_{L^2(\Sigma\times(-\delta,\delta),\CC^4)}
+ E_5(T^N_{s,m,c}) \|\Pi^\perp u\|^2_{L^2(\Sigma\times(-\delta,\delta),\CC^4)},
\end{multline*}
and using the pointwise orthogonality $\big\langle \Pi u(s,\cdot), \Pi^\perp u(s,\cdot)\big\rangle_{L^2((-\delta,\delta),\CC^4)}=0$, $s\in \Sigma$,
one gets
\begin{align*}
\iint_{\Sigma\times(-\delta,\delta)}  (K-M^2-c\delta)\, |u|^2\dd\Sigma\dd t
&= \iint_{\Sigma\times(-\delta,\delta)}  (K-M^2-c\delta)\, |\Pi u|^2\dd\Sigma\dd t\\
&\quad +\iint_{\Sigma\times(-\delta,\delta)}  (K-M^2-c\delta)\, |\Pi^\perp u|^2\dd\Sigma\dd t,
\end{align*}
implying
\begin{multline}\label{eq231}
q^N_{m,\tau,\delta}(u,u) 
\ge (1-c\delta) \iint_{\Sigma\times(-\delta,\delta)} \|\nabla_s u\|^2_{T_s\Sigma \otimes \CC^4}\dd\Sigma\dd t\\
+ \big\langle \Pi u, (K-M^2) \Pi u\big\rangle_{L^2(\Sigma\times(-\delta,\delta),\CC^4)} +(-\mu^2m^2 - b_0 m^2e^{-2\mu m\delta}-c\delta) \|\Pi u\|^2_{L^2(\Sigma\times(-\delta,\delta),\CC^4)}\\
+ \big\langle \Pi^\perp u, (K-M^2) \Pi^\perp u\big\rangle_{L^2(\Sigma\times(-\delta,\delta),\CC^4)} + \Big(\dfrac{b_1^2}{\delta^2}-c\delta\Big) \|\Pi^\perp u\|^2_{L^2(\Sigma\times(-\delta,\delta),\CC^4)}.
\end{multline}
Now we would like to separate the terms with $\Pi u$ and $\Pi^\perp u$ in the first term on the right-hand side.
One has, with the norms and scalar products taken in $L^2\Big(T\Sigma, L^2\big((-\delta,\delta),\CC^4\big)\Big)$,
\begin{equation}
   \label{eq-nu1}
\| \nabla_s u\|^2=\|\nabla_s(\Pi u)\|^2+\|\nabla_s(\Pi^\perp u)\|^2+
2\Re \langle \nabla_s (\Pi u), \nabla_s(\Pi^\perp u)\rangle,
\end{equation}
and
\begin{align*}
\Big\langle \nabla_s (\Pi u), \nabla_s(\Pi^\perp u)\Big\rangle &= \Big\langle \nabla_s \Pi \Pi u, \nabla_s \Pi^\perp \Pi^\perp u \Big\rangle\\
&=\Big\langle \big([\nabla_s,\Pi] + \Pi\nabla_s\big) \Pi u,  \big([\nabla_s,\Pi^\perp] + \Pi^\perp\nabla_s\big) \Pi^\perp u\Big\rangle\\
&=\Big\langle [\nabla_s,\Pi] \Pi u,  [\nabla_s,\Pi^\perp] \Pi^\perp u\Big\rangle
+\Big\langle \Pi\nabla_s \Pi u,  [\nabla_s,\Pi^\perp] \Pi^\perp u\Big\rangle\\
&\quad +\Big\langle [\nabla_s,\Pi] \Pi u,  \Pi^\perp\nabla_s\Pi^\perp u\Big\rangle
+\Big\langle \Pi\nabla_s \Pi u,  \Pi^\perp\nabla_s \Pi^\perp u\Big\rangle\\
&=:J_1+J_2+J_3+J_4.
\end{align*}
Due to the definition of $\Pi$ and $\Pi^\perp$ one has $J_4=0$.
By Lemma~\ref{lem29} we estimate, with some $c_0,c_1>0$ independent of $m$ and $\delta$:
\begin{align*}
|J_1|&\le c_0 \|\Pi u\|\cdot \|\Pi^\perp u\|\le c_0 \delta  \|\Pi u\|^2+ \dfrac{c_0}{\delta}\|\Pi^\perp u\|^2,\\
|J_2|&\le c_1\|\nabla_s \Pi u\|\cdot\|\Pi^\perp u\|
\le c_1\delta \|\nabla_s \Pi u\|^2+ \dfrac{c_1}{\delta}\|\Pi^\perp u\|^2.
\end{align*}
Finally, using the self-adjointness of $\Pi^\perp$ and that by Lemma \ref{lem29} we have $\Pi^\perp [\nabla_s,\Pi] \Pi u \in H^1\Big(T\Sigma, L^2\big((-\delta,\delta),\CC^4\big)\Big)$, we can perform an integration by parts to obtain
\begin{align*}
J_3&=\Big\langle \Pi^\perp [\nabla_s,\Pi] \Pi u, \nabla_s\Pi^\perp u\Big\rangle\\
&=\iint_\Sigma \big\langle \Pi^\perp [\nabla_s,\Pi] \Pi u, \nabla_s \Pi^\perp u\big\rangle_{T_s\Sigma \otimes L^2((-\delta,\delta),\CC^4)} \dd\Sigma(s)\\
&=-\iint_\Sigma \big\langle \ddiv_s \big(\Pi^\perp [\nabla_s,\Pi] \Pi u\big), \Pi^\perp u\big\rangle_{L^2((-\delta,\delta),\CC^4)} \dd\Sigma(s)\\
&=-\Big\langle \ddiv_s \big(\Pi^\perp [\nabla_s,\Pi] \Pi u\big), \Pi^\perp u\Big\rangle,
\end{align*}
which yields
\begin{equation*}
  |J_3| \le \Big\|\ddiv_s \big(\Pi^\perp [\nabla_s,\Pi] \Pi u\big)\Big\|_{L^2(\Sigma\times(-\delta,\delta),\CC^4)}
\cdot\big\|\Pi^\perp u\big\|_{L^2(\Sigma\times(-\delta,\delta),\CC^4)}.
\end{equation*}
Recall that in the local coordinates on $\Sigma$ for a vector field $A=(A_j)$
one has
\[
\ddiv_s A=\sum_{j} \big( \partial_j A_j +\sum_k \Gamma^j_{kj} A_k\big),
\]
with $\Gamma^j_{kj}$ being the Cristoffel symbols depending on the choice of coordinates only.
In our case the $j$-th component of the vector $\Pi^\perp [\nabla_s,\Pi] \Pi u$ can be computed using \eqref{eqn:expcommut} and is
\[
\big(\Pi^\perp [\nabla_s,\Pi] \Pi u\big)_j=\Theta \pi_0^\perp \Theta^* \sum_{k} g^{jk} \big(\partial_k\Theta \cdot \pi_0 \Theta^*+
\Theta \pi_0 \partial_k\Theta^*\big) (\Pi u).
\]
Furthermore, the projector $\pi_0$ does not depend on $s$ while $\Theta$ is $C^2$ in $s$ (see Lemma~\ref{lem-btheta})
and does not depend on $m$ and $\delta$. Therefore, with suitable $c_2>0$ one may estimate
\begin{multline*}
\big\|\ddiv_s \big(\Pi^\perp [\nabla_s,\Pi] \Pi u\big)\big\|_{L^2(\Sigma\times(-\delta,\delta),\CC^4)}\\
\le c_2 \Big(\|\Pi u\|_{L^2(\Sigma\times(-\delta,\delta),\CC^4)} +\|\nabla_s(\Pi u)\|_{L^2(T\Sigma, L^2((-\delta,\delta),\CC^4))}\Big),
\end{multline*}
which gives
$$
|J_3|\le 2 c_2 \delta \Big(\|\Pi u\|^2_{L^2(\Sigma\times(-\delta,\delta),\CC^4)} +\|\nabla_s(\Pi u)\|^2_{L^2(T\Sigma, L^2((-\delta,\delta),\CC^4))}\Big)+ \dfrac{c_2}{\delta}\big\|\Pi^\perp u\big\|^2_{L^2(\Sigma\times(-\delta,\delta),\CC^4)}.
$$
Therefore, from \eqref{eq-nu1} we obtain, with a suitable $c_3>0$,
\begin{align*}
\|\nabla_s u\|_{L^2(T\Sigma, L^2((-\delta,\delta),\CC^4))}^2 &\ge \|\nabla_s (\Pi u)\|^2_{L^2(T\Sigma, L^2((-\delta,\delta),\CC^4))}\\
& \quad-2 \Big|\big\langle \nabla_s (\Pi u),\nabla_s(\Pi^\perp u)\big\rangle_{L^2(T\Sigma, L^2((-\delta,\delta),\CC^4))}\Big|\\
&\ge (1-c_3 \delta)\|\nabla_s (\Pi u)\|^2_{L^2(T\Sigma, L^2((-\delta,\delta),\CC^4))}\\
&\quad -c_3 \delta\|\Pi u\|^2_{L^2(\Sigma, L^2((-\delta,\delta),\CC^4))}
-\dfrac{c_3}{\delta}\|\Pi^\perp u\|^2_{L^2(\Sigma, L^2((-\delta,\delta),\CC^4))}.
\end{align*}
Let us substitute all the estimates obtained into \eqref{eq231}.
One remarks that all terms $\Pi^\perp u$ can be minorated by
\[
\Big(\dfrac{b_1^2}{\delta^2}-c\delta-\dfrac{(1-c\delta)c_3}{\delta}\Big)
\big\| \Pi^\perp u\big\|^2_{L^2(\Sigma\times(-\delta,\delta),\CC^4)}
+\big\langle \Pi^\perp u, (K-M^2) \Pi^\perp u\big\rangle_{L^2(\Sigma\times(-\delta,\delta),\CC^4)}.
\]
Therefore, one can increase the value of $m_0$ such that for $m>m_0$ the term becomes non-negative
(as $\delta$ becomes small).
Therefore, for large $m>m_0$ we may simply estimate
\begin{equation}
    \label{eq-u0}
q^N_{m,\tau,\delta}(u,u)\ge  q_0(\Pi u,\Pi u),
\end{equation}
where $q_0$ is the sesquilinear form in the Hilbert space $\ran \Pi$ defined on $\Pi\big(\dom(q^N_{m,c})\big)$ by
\begin{multline*}
q_0(u,u)=
(1-b\delta) \iint_{\Sigma\times(-\delta,\delta)} \|\nabla_s u\|^2_{T_s\Sigma \otimes \CC^4}\dd\Sigma\dd t\\
+ \big\langle u, (K-M^2) u\big\rangle_{L^2(\Sigma\times(-\delta,\delta),\CC^4)}-(\mu^2m^2 + b m^2e^{-2\mu m\delta}+b\delta) \|u\|^2_{L^2(\Sigma\times(-\delta,\delta),\CC^4)}
\end{multline*}
and $b>0$ is a suitable constant.

Now define a sesquilinear form $q$ on $\ran (\Pi) \times \ran(\Pi^\perp)$ by $q\big( (u,u^\perp),(u,u^\perp)\big)=q_0(u,u)$.
Then the inequality \eqref{eq-u0} takes the form $q_{m, \tau, \delta}^N(u,u)\ge q_0(\Pi u,\Pi u) = q(Uu,Uu)$, where
$U u=(\Pi u,\Pi^\perp u)$. As $U$ is unitary, one has by the min-max principle
$E_j(q^N_{m,\tau,c})\ge E_j(q)$ for all $j$. On the other hand, due to the representation
$q=q_0\oplus 0$ we have $E_j(q)=E_j(q_0)$ for all $j\in\NN$ with $E_j(q_0)<0$.
Therefore, $E_j(q^N_{m,\tau,\delta})\ge E_j(q_0)$ for all $j\in \big\{1,\dots, \cN(q_0,0)\big\}$.
But again due to the form inequality one has $\cN(q^N_{m,\tau,\delta},0)\le \cN(q_0,0)$, therefore,
\[
E_j(q^N_{m,\tau,\delta})\ge E_j(q_0) \text{ for all } j\in\big\{1,\dots,\cN(q^N_{m,\tau,\delta},0)\big\}.
\]

It remains to find a suitable expression for $E_j(q_0)$.
Let $\cH$ be defined by~\eqref{def_Hilbert_space_H}.
Using the representation \eqref{eq-vv3} and choosing a constant $c_m>0$ such that $c_m^2\|\psi_{m,c}\|^2_{L^2(0,\delta)}=1$
one concludes that the map
\[
V: \cH \to \ran (\Pi), \quad (V v)(s,t)=c_m v_\pm(s) \psi\big(|t|\big) \text{ for } \pm t>0,
\]
is unitary, and with the form $\ell^\tau_0$ from \eqref{eq-ll1} we have
\[
q_0(Vv,Vv)=(1-b\delta) \ell^\tau_0(v,v) + \big\langle v, (K-M^2) v\big\rangle_{\cH} +(-\mu^2m^2 - b m^2e^{-2\mu m\delta}-b\delta) \|v\|^2_\cH,
\]
which shows
\[
E_j(q_0)=-\mu^2m^2+E_j\big((1-b\delta)\cL^\tau_0 +K-M^2\big)- b m^2e^{-2\mu m\delta}-b\delta
\]
for all $j\in\NN$ and concludes the proof of this lemma.
\end{proof}

\begin{proof}[Proof of the lower bound in Proposition~\ref{prop-main1}]
It is sufficient to use the estimate of Lemma~\ref{lem-lowb}
in the left-hand inequality of Lemma~\ref{lem-qnd}.
\end{proof}

\section*{Acknowledgments}
Markus Holzmann was supported by the Austrian Agency for International Cooperation in Education and Research (OeAD).
Thomas Ourmi\`eres-Bonafos was supported by a public grant as part of the ``Investissement d'avenir'' project, reference ANR-11-LABX-0056-LMH, LabEx LMH.
Thomas Ourmi\`eres-Bonafos and Konstantin Pankrashkin are supported by the PHC Amadeus 2017--2018 37853TB
funded by the French Ministry of Foreign Affairs and the French Ministry of Higher Education, Research and Innovation.
The authors thank Yuri Kordyukov for comments on a preliminary version of the text.

\end{document}